\documentclass[10pt,reqno]{amsart}
\usepackage{amssymb, hyperref}
\usepackage{color}

\def\norm#1{\left\|#1\right\|}
\def\bra#1{\langle#1\rangle}
\def\wt#1{\widetilde{#1}}
\def\wh#1{\widehat{#1}}
\def\owh#1{\overline{\widehat{#1}}}
\def\set#1{\{#1\}}

\newcommand{\R}{{\mathbb R}}
\newcommand{\C}{{\mathbb C}}
\newcommand{\T}{{\mathbb T}}
\newcommand{\Z}{{\mathbb Z}}
\newcommand{\N}{{\mathcal{N}}}
\newcommand{\ft}{{\mathcal{F}}}
\newcommand{\Sch}{{\mathcal{S}}}

\newcommand{\px}{\partial_x}
\newcommand{\pt}{\partial_t}

\newcommand{\la}{\lambda}

\newcommand{\ep}{\epsilon}

\newcommand{\ol}{\overline}
\newcommand{\ul}{\underline}

\numberwithin{equation}{section}
\newtheorem{theorem}{Theorem}[section]

\newtheorem{proposition}[theorem]{Proposition}
\newtheorem{lemma}[theorem]{Lemma}
\newtheorem{corollary}[theorem]{Corollary}

\newtheoremstyle{definition}{}{}%
     {}
     {}
     {\bfseries}
     {. }
     {0em}
     {}

\theoremstyle{definition}

\newtheorem{remark}[theorem]{Remark}

\begin{document}
\title[Fourth-order cubic NLS]{Periodic fourth-order cubic NLS: Local well-posedness and Non-squeezing property}


\author[C. Kwak]{Chulkwang Kwak}
\email{chkwak@mat.uc.cl}
\address{Facultad de Matem\'aticas Pontificia Universidad Cat\'olica de Chile, Campus San Joaquín. Avda. Vicuña Mackenna 4860, Santiago, Chile}

\begin{abstract}
In this paper, we consider the cubic fourth-order nonlinear Schr\"odinger equation (4NLS) under the periodic boundary condition. We prove two results. One is the local well-posedness in $H^s(\T)$ with $-1/3 \le s < 0$ for the Cauchy problem of the Wick ordered 4NLS. The other one is the non-squeezing property for the flow map of 4NLS in the symplectic phase space $L^2(\T)$. To prove the former we used the ideas introduced in \cite{TT2004} and \cite{NTT2010}, and to prove the latter we used the ideas in \cite{CKSTT2005}. 
\end{abstract}

\thanks{} 

\maketitle

\tableofcontents

\section{Introduction}


In this paper, we consider the following one-dimensional cubic fourth-order Schr\"odinger equation (4NLS) with the periodic boundary condition:
\begin{equation}\label{eq:NLS1}
\begin{cases}
i\pt u + \px^4 u =  \mu|u|^2u,\\
u(0,x) = u_0(x) \in H^s(\T),
\end{cases}
\quad (t,x) \in \R \times \T,
\end{equation}
where $\mu = \pm 1$\footnote{Here $\pm1$ does not affect our analysis, so we only consider the case when $\mu = 1$.}, $\T=\R/(2\pi\Z)$ and $u$ is a complex-valued function. The equation \eqref{eq:NLS1} is known as the biharmonic NLS and for its physical backgrounds see \cite{Karpman1996, KS2000} and the references therein. 
As a physical model, the equation \eqref{eq:NLS1} has many symmetries. The first one is the scaling invariance (although not strict): if $u$ is a solution to \eqref{eq:NLS1}, $u_{\lambda}$ is also a solution to \eqref{eq:NLS1} on the stretched torus $\T_\lambda=\R/(2\pi \lambda^{-1}\Z)$, where
\[u_{\lambda}(t,x) = \lambda^2 u(\lambda^4 t , \lambda x), \quad \lambda >0.\]
Moreover, 
$\norm{u_{0,\lambda}}_{\dot{H}^s(\T_{\lambda})} = \lambda^{s+3/2}\norm{u_0}_{\dot{H}^s(\T)}$,
which shows the scaling critical exponent $s_c = - 3/2$. The equation \eqref{eq:NLS1} also admits important conservation laws, such as mass and Hamiltonian conservation laws: if $u$ is a smooth solution to \eqref{eq:NLS1} then
\begin{equation}\label{eq:mass}
M[u(t)] = \int_{\T} |u|^2 \; dx = M[u_0]
\end{equation}
and
\begin{equation}\label{eq:Hamiltonian}
H[u(t)] = \frac12\int_{\T} |\px^2 u|^2 \; dx - \frac{\mu}{4}\int_{\T} |u|^4 \; dx = H[u_0].
\end{equation}

The fourth-order NLS has been extensively studied. See \cite{FIP2003,BKS2000,GW2002,HHW2006,HHW2007,Segata2006,Pausader2007,Pausader2009-1,Pausader2009-2} and references therein for results of Cauchy problems on $\R^n$, $n\ge1$. Global well-posedness of \eqref{eq:NLS1} in $L^2$ and non-existence below $L^2$ were proved in Appendix A in \cite{OT2016} where the quasi-invariant measure for \eqref{eq:NLS1} was studied. This well-/ill-posedness result is analogue to the classical cubic NLS ($\partial_x^4$ replaced by $-\partial_x^2$ in \eqref{eq:NLS1}). The well-posedness was obtained via Bourgain's space method \cite{Bourgain1993} and the ill-posedness was based on the study of wick ordered equations as in \cite{GO2015}.

\subsection{Wick ordered fourth-order cubic NLS}\label{sec:WNLS}
For the global solution $u \in C(\R;L^2(\T))$ to the equation \eqref{eq:NLS1}, we consider the gauge transformation $\mathcal{G}$ defined by
\begin{equation}\label{eq:gauge}
v=\mathcal{G}[u](t) := e^{2it\oint |u(t,x)|^2 \; dx}u(t),
\end{equation}
where 
\[\oint f \; dx := \frac{1}{2\pi}\int_{\T} f \; dx.\]
It is noticed that the gauge transformation $\mathcal{G}$ defined in \eqref{eq:gauge} depends only on $u_0$ because of the mass conservation, as long as $u_0 \in L^2(\T)$. Besides, \eqref{eq:gauge} reduces the equation \eqref{eq:NLS1} to the following \emph{Wick ordered cubic fourth-order NLS} (4WNLS):
\begin{equation}\label{eq:WNLS1}
\begin{cases}
i\pt v + \px^4 v =  \left(|v|^2 - 2\oint|v|^2 \;dx\right)v,\\
v(0,x) = v_0(x),
\end{cases}
\quad (t,x) \in \R \times \T.
\end{equation}
Compared to \eqref{eq:NLS1}, the equation \eqref{eq:WNLS1} contains fewer resonant terms (see Section \ref{sec:preliminaries}). Moreover, in $L^2$ or $H^s$ with $s\geq 0$ the gauge transformation $\mathcal{G}$ is well defined and invertible, and thus one can freely transfer the results between the two equations. 

The key ideas in \cite{GO2015} which were later employed in \cite{OT2016} are: below $L^2$ the equation \eqref{eq:WNLS1} behaves better than \eqref{eq:NLS1} while $\mathcal{G}$ is not defined. One can combine the good behaviour of \eqref{eq:WNLS1} and the bad property of $\mathcal{G}$ to show some strong ill-posedness for \eqref{eq:NLS1} below $L^2$. To show the non-existence, one only needs the existence for the wick ordered equation \eqref{eq:WNLS1} below $L^2$.  However, we can prove well-posedness below $L^2$.
The following is the primary result in this paper.

\begin{theorem}\label{thm:main}
Let $-1/3 \le s < 0$ and $R>0$. For any $u_0 \in B_R=\{f\in H^s(\T): \norm{f}_{H^s}\leq R\}$, there exist $T = T(R)>0$ and a unique solution to \eqref{eq:WNLS1} on $[-T, T]$ satisfying\footnote{The solution space $Y_T^{s,\frac12}$ is defined in Section \ref{sec:preliminaries} .}
\[u \in C([-T,T];H^s(\T)) \cap Y_T^{s,\frac12}.\]
Moreover, the solution map $u_0 \mapsto u$ is continuous from $B_R$ to $C([-T,T];H^s(\T))$.
\end{theorem}

In the proof of Theorem \ref{thm:main}, we follow the ideas in \cite{TT2004} and \cite{NTT2010}. In these works, low-regularity well-posedness for periodic mKdV was proved by using a new $X^{s,b}$-type space associated to the initial data and some energy-type estimates. This idea has been also applied to the well-posedness of NLS with the third order dispersion by Miyaji and Tsutsumi \cite{MT2017}. We would like to point two different views in this paper out  from the previous works. The main task in this part is to control the resonant terms
\begin{equation}\label{eq:task1}
\left||\wh{u}(n)|^2 - |\wh{u}_0(n)|^2\right|
\end{equation}
for \emph{a priori bound} of a single solution, and 
\begin{equation}\label{eq:task2}
\left||\wh{u}(n)|^2 - |\wh{v}(n)|^2\right|
\end{equation}
for the uniqueness part (see Section \ref{sec:preliminaries}, \eqref{eq:gap} and \eqref{eq:diff}). In previous works \cite{TT2004, NTT2010, MT2017}, authors control the $\ell^1$-norm of \eqref{eq:task1}, while we control the $\ell^{\infty}$-norm of \eqref{eq:task1}. We realized the advantage from the gap between $\ell^{\infty}$-norm and the $\ell^1$-norm,
and it enabled us to cover the end point regularity (see Remark \ref{rem:threshold}). This way seems to be applicable to \cite{TT2004, NTT2010} and \cite{MT2017} for covering the endpoint regularity $s = 1/3$ and $s = -1/6$, respectively. The idea can be also seen in \cite{MPV2016} where the unconditional well-posedness of modified KdV is shown for $s \ge 1/3$. Moreover, we perceived that it is nontrivial to apply the estimation of \eqref{eq:task1} to the estimation of \eqref{eq:task2}. Indeed, in view of the proof of Proposition \ref{prop:main}, the $H^s$ energy estimate of a single solution is necessary to control the contribution of the boundary term arising in the normal form reduction method. However, in the estimate of \eqref{eq:task2}, not only a single solution but also the difference of two solutions appear in the contribution of the boundary term, while one cannot obtain the $H^s$-energy estimate of the difference of two solutions for $s \ge -1/3$. We took a trick to resolve this nontrivial issue, and hence we closed the estimate of \eqref{eq:task2}. See Remark \ref{rem:nontrivial}, the proof of Proposition \ref{prop:main2} and Lemma \ref{lem:sym} for the details.

\begin{remark}
The fourth-order cubic NLS \eqref{eq:NLS1} was already known to be ill-posed in $H^s$ for $s<0$ in the sense of the failure of uniform continuity of solution map by following the argument in Burq-G\'erard-Tzvetkov \cite{BGT2002} and Christ-Colliander-Tao \cite{CCT2003} (see Lemma 6.16 in \cite{OT2016}). Moreover, in view of Remark 1.4 in \cite{CO2012} (or by simple calculation), we can know that the example chosen in \cite{OT2016} for the focusing (or defocusing) 4NLS \eqref{eq:NLS1} is exactly the solution to the defocusing (or focusing) 4WNLS \eqref{eq:WNLS1}. Hence, Lemma 6.16 in \cite{OT2016} exactly shows the failure of uniform continuity of the solution map of \eqref{eq:WNLS1} below $L^2(\T)$. 
\end{remark}

Recently, we learned some results concerning well-posedness of \eqref{eq:WNLS1}. Oh and Wang \cite{OW2017} proved the global existence in $H^s$, $s > -9/20$ by using the short time Fourier restriction norm method, and showed the enhanced uniqueness in $H^s$, $s>-1/3$ by using the infinite iteration of normal form reductions. Also Oh, Tzvetkov and Wang \cite{OTW2017} proved almost surely global well-posedness in $H^s$, $s < -1/2$.

\subsection{Non-squeezing property of NLS}\label{subsec:Nonsq prop}

The non-squeezing phenomena is an important property of the Hamiltonian system. It states that one cannot embed a ball into a cylinder via a symplectic map unless the radius of the ball is less than or equal to the radius of the cylinder although the symplectic map is area-preserving. The finite-dimensional non-squeezing theorem was established by Gromov \cite{Gromov1985}. The extension to the infinite-dimensional setting was initiated by Kuksin \cite{Kuksin1995} for certain equations, whose flow maps consist of linear and compact smooth operators. Thereafter, the study on the symplectic property of the flow map, which has non-compact nonlinearity, was extended by several researchers. Bourgain \cite{Bourgain1994} presented the non-squeezing analysis for the 1-dimensional cubic NLS on the symplectic phase space $L^2(\T)$, and Colliander, Keel, Staffilani, Takaoka and Tao \cite{CKSTT2005} showed the non-squeezing property for the KdV equation on $H^{-\frac12}(\T)$. Recently, the non-squeezing properties were proved for many other equations, such as the BBM equation on $H^{\frac12}(\T)$ by Roum\'egoux \cite{Roumegoux2010}, the Klein-Gordon equation on $\mathcal{H}^{\frac12}(\T^3) = H^{\frac12}(\T^3) \times H^{-\frac12}(\T^3)$ by Mendelson \cite{Mendelson2014}, the coupled KdV-type system without the Miura transform on $H^{-\frac12}(\T) \times H^{-\frac12}(\T)$ by Hong and Kwon \cite{HK2015} and the higher-order KdV equation with the nonlinearity of the form $u\px u$ on $H^{-\frac12}(\T)$ by Hong and the author \cite{HK2016}. Very recently, Killp, Visan and Zhang \cite{KVZ2016-1, KVZ2016-2} introduced the first symplectic non-squeezing result for a Hamiltonian PDE in infinite volume, in particular, the cubic NLS on $L^2(\R^2)$ and $L^2(\R)$.

The second result in this paper is the non-squeezing property of \eqref{eq:NLS1}. The equation \eqref{eq:NLS1} can be written as the Hamiltonian form
\[\dot u(t)=i\nabla_u H[u]\]
on the phase space $L^2(\T)$ with the symplectic form $\omega_0$ defined by
\[\omega_0(u,v) =-\mbox{Im}\int_\T u\bar v dx.\]
The flow map $\Phi(t)$ of \eqref{eq:NLS1}, particularly, is symplectomorphism on $L^2(\T)$. We prove
\begin{theorem}\label{thm:Nonsqueezing thm}
Let $ 0 < r< R$, $u_* \in L^2(\T)$, $n_0 \in \Z$, $z \in \C$ and $T>0$. Then there exists a global $L^2$-solution $u$ ($:=\Phi(t)u_0$) to \eqref{eq:NLS1} such that
\begin{equation*}
\norm{u_0 - u_*}_{L^2} \leq R
\end{equation*}
and 
\begin{equation*}
\left| \ft_x[\Phi(T)u_0](n_0) - z\right| >r,
\end{equation*}
where $\ft_x$ and $\Phi$ are the spatial Fourier transform and the flow map of \eqref{eq:NLS1}, respectively.
\end{theorem}

In order to prove Theorem \ref{thm:Nonsqueezing thm}, our main task is to prove the approximation of a truncated flow to the original flow. Following Bourgain's approach, we consider the truncated equation:
\begin{equation}\label{eq:truncated equation}
\left\{\begin{array}{ll}
i\pt u + \px^4 u =  P_{\le N}\left(\mu|u|^2u\right), \hspace{1em} &(t,x) \in \R \times \T, \\
u(x,0) = u_0(x), &u_0 \in P_{\leq N}L^2(\T),
\end{array}
\right.
\end{equation}
where $P_{\leq N}$ is the Fourier projection operator for the spatial frequency defined as follows: for each dyadic number $N$,
\begin{equation}\label{eq:projection op}
\begin{aligned}
&\widehat{P_Nu}(n) := 1_{N \leq |n| <2N}(n) \widehat{u}(n), \\
&\widehat{P_{\leq N} u}(n) := 1_{|n| \leq N}(n) \widehat{u}(n), \\
&\widehat{P_{\geq N} u}(n) := 1_{|n| \geq N}(n) \widehat{u}(n),
\end{aligned}
\end{equation}
where $1_{\Omega}$ is a characteristic function on $\Omega$. The equation \eqref{eq:truncated equation} is a finite-dimensional Hamiltonian equation with Hamiltonian \eqref{eq:Hamiltonian} in $P_{\leq N}L^2$. Thus, the flow map of \eqref{eq:truncated equation} denoted by $\Phi^N(t)$ is a finite dimensional symplectic map, so we can apply Gromov's theorem directly to this map (see Lemma \ref{lem:Nonsqueezing of trun. flow}). Then by showing the low frequency approximation in the strong $L^2$-topology (see Proposition \ref{prop:approximation}), we can obtain the non-squeezing property of \eqref{eq:NLS1}. 

\textbf{Organization of paper.} The rest of the paper is organized as follows: In Section \ref{sec:preliminaries}, we summarize some notations and define function spaces. We also provide proofs of some important lemmas, which are used to prove both Theorems \ref{thm:main} and \ref{thm:Nonsqueezing thm}. In Section \ref{sec:main}, we mainly prove the local well-posedness of the Wick ordered fourth-order NLS \eqref{eq:WNLS1} below $L^2(\T)$. Finally, we prove the non-squeezing property of \eqref{eq:NLS1} in the symplectic phase space $L^2(\T)$ in Section \ref{sec:nonsqueezing}.

\textbf{Acknowledgments.}
The author would like to appreciate Prof. Zihua Guo for helpful discussion and encouragement. Also, the author thank Prof. Tadahiro Oh for pointing out an unclear portion in the proof of Proposition 3.5. C. Kwak is supported by FONDECYT de Postdoctorado 2017 Proyecto No. 3170067.

\section{Preliminaries}\label{sec:preliminaries}
For $x, y \in \R_+$, $x \lesssim y$ denotes $x \le Cy$ for some $C >0$ and $x \sim y$ means $x \lesssim y$ and $y \lesssim x$. Also, $x \ll y$ denotes $x \le cy$ for a small positive constant $c$. Let $a_1,a_2,a_3 \in \R_{+}$. The quantities $a_{max} \ge a_{med} \ge a_{min}$ can be defined to be the maximum, median and minimum values of $a_1,a_2,a_3$, respectively.

For $f \in \Sch '(\R \times \T) $ we denote by $\wt{f}$ or $\ft (f)$ the Fourier transform of $f$ with respect to both spatial and time variables,
\[\wt{f}(\tau,n)=\frac{1}{2\pi}\int_{\R}\int_{0}^{2\pi} e^{-ixn}e^{-it\tau}f(t,x)\; dxdt .\]
Moreover, we use $\ft_x$ (or $\wh{\;}$ ) and $\ft_t$ to denote the Fourier transform with respect to space and time variables, respectively.

We first observe the Fourier coefficient of \eqref{eq:NLS1} in terms of the spatial variable at the frequency $n$ as follows:
\begin{equation}\label{eq:NLS2}
\pt \wh{u}(n) -in^4 \wh{u}(n) = -i\sum_{n_1 - n_2 + n_3 = n}\wh{u}(n_1)\owh{u}(n_2)\wh{u}(n_3).  
\end{equation}
From the cubic resonant relation in the nonlinear term, we have (see Lemma 3.1 in \cite{OT2016} for the proof)
\begin{equation}\label{eq:resonant function}
\begin{aligned}
H&:= H(n_1,n_2,n_3) = n_1^4 - n_2^4 + n_3^4 - (n_1 - n_2 + n_3)^4\\
&=(n_1-n_2)(n_2-n_3)\left(n_1^2 + n_2^2 + n_3^2 + (n_1-n_2+n_3)^2 +2(n_1+n_3)^2 \right)
\end{aligned}
\end{equation}
and we can know that non-trivial resonances appear only when $n_1 = n_2$ or $n_2 = n_3$. Thus, we can rewrite \eqref{eq:NLS2} by
\[\pt \wh{u}(n) -in^4 \wh{u}(n) = i  |\wh{u}(n)|^2\wh{u}(n) -2i   \norm{u}_{L^2}^2\wh{u}(n) -i \sum_{\N_n}\wh{u}(n_1)\owh{u}(n_2)\wh{u}(n_3),\]
where $\N_n$ is the cubic non-resonant set of frequencies at the frequency $n$
\[\N_n = \set{(n_1,n_2,n_3) \in \Z^3 : n_1-n_2+n_3=n, \; (n_1-n_2)(n_2-n_3) \neq 0}.\]
Similarly, the equation \eqref{eq:WNLS1} (in terms of $u$ instead of $v$) can be rewritten as
\begin{equation}\label{eq:WNLS2}
\pt \wh{u}(n) -in^4 \wh{u}(n) = i  |\wh{u}(n)|^2\wh{u}(n) -i \sum_{\N_n}\wh{u}(n_1)\owh{u}(n_2)\wh{u}(n_3).
\end{equation}
Then we see the only resonant term in \eqref{eq:WNLS2} is $i|\wh{u}(n)|^2\wh{u}(n)$ while the worst resonant term $-2i   \norm{u}_{L^2}^2\wh{u}(n)$ is removed by Gauge transform \eqref{eq:gauge}.
We define 
\begin{equation}\label{eq:res. nonl.}
\N_{R}(u,v,w)(n) := \ft_x^{-1}[i\wh{u}(n)\ol{\wh{v}}(n)\wh{w}(n) - 2i (\sum_{n_1 \in \Z} \wh{u}(n_1)\ol{\wh{v}}(n_1))\wh{w}(n)]
\end{equation}
and 
\begin{equation}\label{eq:nonres. nonl.}
\N_{NR}(u,v,w)(n) := \ft_x^{-1}[-i \sum_{\N_n}\wh{u}(n_1)\owh{v}(n_2)\wh{w}(n_3)].
\end{equation}
We simply write $\N_R(u)$ and $\N_{NR}(u)$ for $\N_R(u,u,u)$ and $\N_{NR}(u,u,u)$, respectively. 

In \cite{TT2004, NTT2010} in the context of modified KdV equation, the authors modified the linear propagator by choosing the first approximation evolution operator with the oscillation factor $e^{itn|\wh{u}(0,n)|^2}$ in order to weaken the nonlinear perturbation $in|\wh{u}(n)|^2\wh{u}(n)$. In this paper, we use the same ideas as in \cite{TT2004, NTT2010} to weaken the resonant term 
\begin{equation}\label{eq:resonant term}
i  |\wh{u}(n)|^2\wh{u}(n).
\end{equation}
Precisely, in order to remove the non-trivial resonant term \eqref{eq:resonant term}, one needs to use the evolution operator $\mathcal{V}(t)$ as
\[\mathcal{V}(t)u := \frac{1}{\sqrt{2\pi}} \sum_{n \in \Z} e^{i(nx + tn^4 +  \int_0^t |\wh{u}(s,n)|^2 \; ds)}\wh{u}(n).\]
However, it is quite difficult to treat the nonlinear oscillation factor $e^{i \int_0^t |\wh{u}(s,n)|^2 \; ds}$. Hence by choosing the first approximation of $\mathcal{V}(t)$
\[\mathcal{W}(t)u := \frac{1}{\sqrt{2\pi}} \sum_{n \in \Z} e^{i(nx + tn^4 +  t |\wh{u}_0(n)|^2)}\wh{u}(n),\]
one can reduce \eqref{eq:WNLS2} to
\begin{equation}\label{eq:WNLS3}
\pt \wh{u}(n) -i\mu(n) \wh{u}(n) = i  (|\wh{u}(n)|^2 - |\wh{u}_0(n)|^2)\wh{u}(n) -i \sum_{\N_n}\wh{u}(n_1)\owh{u}(n_2)\wh{u}(n_3),
\end{equation}
where 
\begin{equation}\label{eq:modified linear operator}
\mu(n) = n^4 + |\wh{u}_0(n)|^2.
\end{equation}
The key observation as in \cite{TT2004} is that the term $(|\wh{u}(n)|^2 - |\wh{u}_0(n)|^2)\wh u(n)$ has smoothing effects (see Corollary \ref{cor:main}). Indeed, from the equation \eqref{eq:WNLS3}, this term equals to 
\[2i \mbox{Im} \left[ \int_0^t \sum_{\N_n} \wh{u}(s,n_1)\owh{u}(s,n_2)\wh{u}(s,n_3)\owh{u}(s,n) \; ds\right]\wh{u}(n).\]
The smoothing effect is due to the highly non-resonant structure. 

\begin{remark}\label{rem:negligible factor}
In view of \eqref{eq:resonant function}, we know that $H$ is roughly bounded below by $\max(n_1^2,n_2^2,n_3^2,n^2)$ over the non-resonant set $\N_n$. Moreover, if $u_0\in H^s$ for $-1<s<0$ then $|\wh{u}_0(n)|^2$ is bounded above by $n^2$. Hence the dominant factor of $G$ (the resonance function for \eqref{eq:WNLS3}, see \eqref{eq:resonant relation2} below) is $H$ and we have all the same estimates if we ignore $|\wh{ u}_0(n)|^2$.
\end{remark}

We now define the standard $X^{s,b}$ space, 
\[\begin{aligned}
&X^{s,b} = \set{v \in \Sch'(\R \times \R) : v(t,x) = v(t, x + 2\pi), \norm{v}_{X^{s,b}} < \infty},\\
&\quad \mbox{with}\quad\norm{v}_{X^{s,b}}^2 = \sum_{n \in \Z} \int_{\R} \bra{n}^{2s} \bra{\tau-n^4}^{2b}|\wt{v}(\tau,n)|^2 \; d\tau,
\end{aligned}\]
where $\bra{\cdot} = (1+|\cdot|^2)^{\frac12}$. The $X^{s,b}$ space was first introduced in its current form by Bourgain \cite{Bourgain1993} and further developed by Kenig, Ponce and Vega \cite{KPV1996} and Tao \cite{Tao2001}. 

In view of \eqref{eq:WNLS3}, we modify the $X^{s,b}$ space corresponding to the linear operator $\mu(-i\px)$ and we denote by $Y^{s,b}$: For $s,b \in \R$, 
\begin{equation}\label{eq:function space}
\begin{aligned}
&Y^{s,b} = \set{v \in \Sch'(\R \times \R) : v(t,x) = v(t, x + 2\pi), \norm{v}_{Y^{s,b}} < \infty},\\
&\norm{v}_{Y^{s,b}}^2 = \sum_{n \in \Z} \int_{\R} \bra{n}^{2s} \bra{\tau-\mu(n)}^{2b}|\wt{v}(\tau,n)|^2 \; d\tau,
\end{aligned}
\end{equation}
where $\mu(n)$ is defined as in \eqref{eq:modified linear operator}. Note that the function space $Y^{s,b}$ is dependent on a given initial data. For $T > 0$, we define the standard time restriction function space of \eqref{eq:function space}
\[Y_T^{s,b} = \set{v \in \mathcal{D}'((-T,T) \times \T) : \norm{v}_{Y_T^{s,b}} < \infty},\]
equipped with the norm
\[\norm{v}_{Y_T^{s,b}} = \inf \set{\norm{w}_{Y^{s,b}} : w \in Y^{s,b}, \; w = v \; \mbox{on} \; (-T, T)}.\]

Moreover, in order to investigate the non-squeezing property of \eqref{eq:NLS1}, we define the standard solution space $Z^{s,\frac12}$, $s \ge 0$, for the periodic problem as follows:
\[\begin{aligned}
&Z^{s,\frac12} = \set{v \in \Sch'(\R \times \R) : v(t,x) = v(t, x + 2\pi), \norm{v}_{Z^{s,\frac12}} < \infty},\\
&\quad \mbox{with}\quad\norm{v}_{Z^{s,\frac12}} := \norm{v}_{X^{s,\frac12}} + \norm{\bra{n}^{s}\wt{v}}_{\ell_n^2L_{\tau}^1}.
\end{aligned}\]
We also need the norm for the nonlinear term, which corresponds to $Z^{s,\frac12}$-norm
\[\norm{v}_{Z^{s,-\frac12}} := \norm{v}_{X^{s,-\frac12}} + \norm{\frac{\bra{n}^{s}}{\bra{\tau - n^4}}\wt{v}}_{\ell_n^2L_{\tau}^1}.\]
The following estimates are well-known (see \cite{Bourgain1993, KPV1996, CKSTT2003, CKSTT2004, CKSTT2005} and references therein) facts:
\begin{itemize}
\item (Embedding property) $\displaystyle \norm{u}_{C_tH^s} \lesssim \norm{u}_{Z^{s,\frac12}}$.
\item (Linear estimate) $\displaystyle \norm{u}_{Z^{s,\frac12}} \lesssim \norm{u_0}_{H^s} + \norm{|u|^2u}_{Z^{s,-\frac12}}$.
\end{itemize}  
The following lemma is the $L_{t,x}^4$ Strichartz estimate in the periodic setting. It was first introduced and proved by Bourgain \cite{Bourgain1993} in order to show the local and global well-posedness of periodic NLS and (generalized) KdV equations. Moreover, the argument and proof are further improved by Tao \cite{Tao2001, Tao2006}.
\begin{lemma}[$L^4$ Strichartz estimate]\label{lem:L4}
Let $s > -1/2$ and assume that $u_0 \in H^{s}(\T)$. Then, for $b > 5/16$, we have
\begin{equation}\label{eq:L4}
\norm{f}_{L^4(\R\times\T)} \lesssim \norm{f}_{Y^{0,b}},
\end{equation}
where $Y^{0,b}$ is defined as in \eqref{eq:function space} and the implicit constant depends only on $\|u_0\|_{H^s}$, $s$ and $b$.
\end{lemma}

\begin{proof}
We follow the argument in \cite{TT2004,KPV1996}. For a given $\eta > 0$, since $u_0 \in H^s(\T)$, we can find a positive integer $N=N(\|u_0\|_{H^s},s)$ such that
\begin{equation}\label{eq:initial restriction}
|n|^{-1}|\wh{u}_0(n)|^2 \le |n|^{-1-2s}|n|^{2s}|\wh{u}_0(n)|^2<\eta \hspace{1em} \mbox{when} \hspace{1em} |n| \ge N.  
\end{equation}
From the fact that $\norm{f}_{L^4}^4 = \norm{|f|^2}_{L^2}^2$, it is enough to bound $L^2$ norm of $|f|^2$. 
We split $f$ into two parts as follows:
\[f = f_{high} + f_{low}, \hspace{1em} \wh{f}_{high}(n) = 0 \hspace{0.7em} \mbox{for} \hspace{0.5em} |n| < N \hspace{1em} \mbox{and} \hspace{1em} \wh{f}_{low}(n) = 0 \hspace{0.7em} \mbox{for} \hspace{0.5em} |n| \ge N.\]
Since $|f|^2 = |f_{high}|^2+f_{high}\overline{f_{low}}+\overline{f_{high}}f_{low}+|f_{low}|^2$ and $\norm{f\overline{g}}_{L^2}=\norm{\overline{f}g}_{L^2}$, it suffices to estimate the $L^2$ norms of $|f_{high}|^2$, $f_{high}\overline{f_{low}}$ and $|f_{low}|^2$ terms.
We only estimate $L^2$ norm of $|f_{high}|^2$, since the other terms can be easily controlled.\footnote{In view of \eqref{eq:M} below, we can easily show $M < \infty$ when one of or both $|n_1|$ and $|n-n_1|$ are bounded.} In this case, we may assume that 
\[\wt{f}(\tau,n) = 0, \hspace{1em} \mbox{for} \hspace{0.5em} \tau \in \R, \; |n| < N.\]
By Plancherel's theorem, we get
\[\begin{aligned}
\norm{|f|^2}_{L^2}^2 &\lesssim \sum_{n \in \Z} \int_{\R} \left|\sum_{n_1 \in \Z} \int_{\R} \wt{f}(\tau-\tau_1, n-n_1)\wt{\ol{f}}(\tau_1, n_1) \; d\tau_1 \right|^2 \; d\tau\\
&=\int_{\R} \left|\sum_{n_1 \in \Z} \int_{\R} \wt{f}(\tau-\tau_1, -n_1)\wt{\ol{f}}(\tau_1, n_1) \; d\tau_1 \right|^2 \; d\tau\\
&+\sum_{n \neq 0} \int_{\R} \left|\sum_{n_1 \in \Z} \int_{\R} \wt{f}(\tau-\tau_1, n-n_1)\wt{\ol{f}}(\tau_1, n_1) \; d\tau_1 \right|^2 \; d\tau\\
&=I + II.
\end{aligned}\]

For the term $I$, since 
\begin{align*}
&\Big|\int_{\R}\wt{f}(\tau-\tau_1, -n_1)\wt{\ol{f}}(\tau_1, n_1) \; d\tau_1\Big| \\
\lesssim& \Big( \int_{\R}\bra{\tau-\tau_1-\mu(-n_1)}^{-2b}\bra{\tau_1+\mu(n_1)}^{-2b} \; d\tau_1\Big)^{\frac12}\\
&\hspace{1em}\times \Big(\int_{\R}\bra{\tau-\tau_1-\mu(-n_1)}^{2b}|\wt{f}(\tau-\tau_1, -n_1)|^2\bra{\tau_1+\mu(n_1)}^{2b}\wt{\ol{f}}(\tau_1, n_1)|^2 \; d\tau_1 \Big)^{\frac12}
\end{align*}
and 
\[\int_{\R}\bra{\tau-\tau_1-\mu(-n_1)}^{-2b}\bra{\tau_1+\mu(n_1)}^{-2b} \; d\tau_1 < \infty\]
whenever $b > 1/4$, we get by the Minkowski inequality and the Cauchy-Schwarz inequality that
\[\begin{aligned}
I &\lesssim \int_{\R} \left|\sum_{n_1 \in \Z} \Big(\int_{\R}\bra{\tau-\tau_1-\mu(-n_1)}^{2b}|\wt{f}(\tau-\tau_1, -n_1)|^2\bra{\tau_1+\mu(n_1)}^{2b}\wt{\ol{f}}(\tau_1, n_1)|^2 \; d\tau_1 \Big)^{\frac12} \right|^2 \; d\tau\\
&\lesssim \left|\sum_{n_1 \in \Z} \Big(\int_{\R^2}\bra{\tau-\tau_1-\mu(-n_1)}^{2b}|\wt{f}(\tau-\tau_1, -n_1)|^2\bra{\tau_1+\mu(n_1)}^{2b}\wt{\ol{f}}(\tau_1, n_1)|^2 \; d\tau d\tau_1 \Big)^{\frac12} \right|^2\\
&\lesssim \norm{f}_{Y^{0,b}}^4.
\end{aligned}\]

For the term $II$, from the fact that
\[\norm{fg}_{L^2} = \norm{\ol{fg}}_{L^2} = \norm{f\ol{g}}_{L^2} = \norm{\ol{f}g}_{L^2},\]
we may assume further that $n_1\geq N, n-n_1\geq N$ in the summation of the integrand. 
As the term $I$, we have
\[\begin{aligned}
II &\lesssim \sum_{n \neq 0}\int_{\R} \Bigg[\Big( \sum_{n_1\in\Z}\int_{\R}\bra{\tau-\tau_1-\mu(n-n_1)}^{-2b}\bra{\tau_1+\mu(n_1)}^{-2b} \; d\tau_1\Big)^{\frac12}\\
&\hspace{1em}\times \Big(\sum_{n_1\in\Z}\int_{\R}\bra{\tau-\tau_1-\mu(n-n_1)}^{2b}|\wt{f}(\tau-\tau_1, n-n_1)|^2\bra{\tau_1+\mu(n_1)}^{2b}\wt{\ol{f}}(\tau_1, n_1)|^2 \; d\tau_1 \Big)^{\frac12}\Bigg]^2 \; d\tau\\
&\lesssim M\norm{f}_{Y^{0,b}}^4,
\end{aligned}\]
where
\begin{equation}\label{eq:M-1} 
M = \sup_{(\tau,n) \in \R \times \Z \setminus \set{0}} \Big( \sum_{\substack{n_1 \ge N \\ n-n_1 \ge N}}\int_{\R}\bra{\tau-\tau_1-\mu(n-n_1)}^{-2b}\bra{\tau_1+\mu(n_1)}^{-2b} \; d\tau_1\Big)^{\frac12}.
\end{equation}
Hence, it suffices to show $M < \infty$ whenever $b > 5/16$.

By a simple calculation
\[\int_{\R} \bra{a}^{-\alpha}\bra{b-a}^{-\alpha} \; da \lesssim \bra{b}^{1-2\alpha}\]
for $1/2 < \alpha < 1$, we have 
\begin{equation}\label{eq:M-2}
\eqref{eq:M-1} \lesssim \sup_{(\tau,n) \in \R \times \Z \setminus \set{0}} \Big( \sum_{\substack{n_1 \ge N \\ n-n_1 \ge N}}\bra{\tau-\mu(n-n_1)+\mu(n_1)}^{1-4b}\Big)^{\frac12}.
\end{equation}
Fix $(\tau, n) \in \R \times \Z \setminus \set{0}$. We first investigate the terms inside the brackets on the right-hand side of \eqref{eq:M-2}. A simple calculation yields
\[\begin{aligned}
\tau-\mu(n-n_1)+\mu(n_1) &= \tau-[(n-n_1)^4 + |\wh{u}_0(n-n_1)|^2] + n_1^4 + |\wh{u}_0(n_1)|^2\\
&= 4n\left(n_1^3 - \frac32nn_1^2 + n^2n_1 +\frac{1}{4n}(\tau - n^4)\right)+A,
\end{aligned}\]
where 
\[A = -|\wh{u}_0(n-n_1)|^2 + |\wh{u}_0(n_1)|^2.\]
Let $F(n_1) := n_1^3 - \frac32nn_1^2 + n^2n_1 + \frac{1}{4n}(\tau - n^4)$. Since
\[\partial_{n_1}F(n_1) = 3n_1^2 - 3n_1n + n^2 = (\sqrt{3}n_1 - \frac{\sqrt{3}}{2}n)^2 + \frac14n^2 > \frac{1}{4}n^2,\]
then we know there is only one real root of $F(n_1)=0$ denoted by $\gamma$ and 
\[|\{n_1: |F(n_1)|\lesssim 1\}|\leq 10.\]
We may assume $|F(n_1)|\gg 1$ and factorize $F(n_1)$ as
\[F(n_1)=(n_1-\gamma)(n_1^2 + c_1n_1 + c_2), \hspace{1em} n_1^2 + c_1n_1 + c_2 \neq 0 \mbox{ for all } n_1 \in \Z.\]
From \eqref{eq:initial restriction} under the assumption $n_1, n-n_1 > N$, we get
\[|A| \le |\wh{u}(n_1)|^2 + |\wh{u}(n-n_1)|^2 < n_1\eta + (n-n_1)\eta = n\eta\]
which implies by choosing $\eta > 0$ small enough that
\[|nF(n_1)|\gg |A|\]
and thus get
\begin{equation}\label{eq:M}
\begin{aligned}
\mbox{RHS of } \eqref{eq:M-2} &\lesssim \sup_{(\tau,n) \in \R \times \Z \setminus \set{0}} \Big( \sum_{n \ge n_1 \ge N}(|n_1|\bra{n_1-\gamma}\bra{n_1^2 + c_1n_1 + c_2})^{1-4b}\Big)^{\frac12}\\
&\lesssim \sup_{(\tau,n) \in \R \times \Z \setminus \set{0}} \Bigg[ \Big(\sum_{n \ge n_1 \ge N}\bra{n_1}^{4(1-4b)}\Big)^{\frac14}\\
&\hspace{2em}\times\Big(\sum_{n \ge n_1 \ge N}\bra{n_1-\gamma}^{4(1-4b)}\Big)^{\frac14} \\
&\hspace{2em}\times\Big(\sum_{n \ge n_1 \ge N}\bra{n_1^2 + c_1n_1 + c_2}^{2(1-4b)}\Big)^{\frac12} \Bigg]\\
&< \infty,
\end{aligned}
\end{equation}
since $4(1-4b) < -1$. Therefore, it completes the proof of Lemma \ref{lem:L4}.
\end{proof}

\begin{lemma}[Sobolev embedding]\label{lem:sobolev}
Let $2 \le p < \infty$ and $f$ be a smooth function on $\R \times \T$. Then for $b \ge \frac12 - \frac1p$, we have
\begin{equation}\label{eq:sobolev}
\norm{f}_{L_t^p(H_x^s)} \lesssim \norm{f}_{Y^{s,b}}.
\end{equation}
Similarly we have
\begin{equation}\label{eq:sobolev-1}
\norm{f}_{L_t^p(H_x^s)} \lesssim \norm{f}_{X^{s,b}}.
\end{equation}
\end{lemma}

\begin{proof}
We only prove \eqref{eq:sobolev} and the proof follows directly from the Sobolev embedding in terms of $t$. Indeed, for $S(t)f(t,x) = \ft_x^{-1}[e^{it\mu(n)}\wh{f}(t,n)]$, since $\norm{S(-t)f}_{H_x^s} = \norm{f}_{H_x^s}$, we have
\[\norm{f}_{L_t^p(H_x^s)} = \norm{\norm{S(-t)f}_{H_x^s}}_{L_t^p} \lesssim \norm{\norm{S(-t)f}_{H_x^s}}_{H_t^b}= \norm{f}_{Y^{s,b}}.\]
\end{proof}

The last two lemmas in this section are the main ingredients to show the non-squeezing property of \eqref{eq:NLS1}. Particularly, the factor $N_{max}^{-\frac12+}$ in the first lemma below facilitates that the truncated flow map approximates to the original flow map in $L^2(\T)$.
\begin{lemma}[Trilinear estimate in $Z^{0,\frac12}$]\label{lem:trilinear L2}
Let $N_j$, $j=1,2,3,4$ be dyadic numbers. Let $u_j = P_{N_j}u$ and $|n_j| \sim N_j$, $j=1,2,3$. Then, we have
\begin{equation}\label{eq:bilinear-1}
\norm{P_{N_4}\N_{NR}(u_1,u_2,u_3)}_{X^{0,-\frac12}} \lesssim N_{max}^{-\frac12+}\norm{u_1}_{X^{0,\frac12}}\norm{u_2}_{X^{0,\frac12}}\norm{u_3}_{X^{0,\frac12}},
\end{equation}
and
\begin{equation}\label{eq:bilinear-2}
\norm{\bra{\tau_4 -n_4^4}^{-1}\ft[P_{N_4}\N_{NR}(u_1,u_2,u_3)]}_{\ell_{n_4}^2L_{\tau_4}^1} \lesssim N_{max}^{-\frac12+}\norm{u_1}_{X^{0,\frac12}}\norm{u_2}_{X^{0,\frac12}}\norm{u_3}_{X^{0,\frac12}},
\end{equation}
where $P_N$ is defined as in \eqref{eq:projection op}.
\end{lemma}

\begin{proof}
We first estimate \eqref{eq:bilinear-1}. Let $\la_j = \tau_j - n_j^4$, $j=1,2,3,4$. From the definition of $X^{s,b}$ norm and the duality, it suffices to show for \eqref{eq:bilinear-1} that
\begin{equation}\label{eq:bilinear-1.1}
\begin{aligned}
\Bigg|\sum_{n_4, \N_{n_4}}\int\limits_{\tau_1 - \tau_2 + \tau_3 = \tau_4} \frac{N_{max}^{\frac12 - \epsilon}}{\prod_{j=1}^{4} \bra{\la_j}^{\frac12}}\wt{u}_1(\tau_1,n_1)\ol{\wt{u}}_2(\tau_2,n_2)\wt{u}_3(\tau_3,n_3)\ol{\wt{u}}_4(\tau_4,&n_4) \; d\tau_1d\tau_2d\tau_3 \Bigg| \\ 
&\lesssim \prod_{j=1}^{3} \norm{u_j}_{L_{t,x}^2},
\end{aligned}
\end{equation}
where $u_4 = P_{N_4}u$ with $\norm{u}_{L_{t,x}^2} \le 1$ and $0 < \epsilon < 1/2$. 

Without loss of generality we may assume that $|\la_1| \le |\la_2| \le |\la_3| \le |\la_4|$.\footnote{In view of the proof, changing the order of modulations does not affect our proof.} Then, we know from \eqref{eq:resonant function} that $|\la_4| \gtrsim |n_1-n_2|(n^{\ast})^2$, for $n^{\ast} = \max(|n_1|,|n_2|,|n_3|,|n_4|)$. Let $\wt{f}_j(\tau_j,n_j) = \bra{\tau_j - n_j^4}^{-\frac12}|\wt{u}_j(\tau_j,n_j)|$ for $j=1,2,3$, $\wt{g}_4(\tau_4,n_4) = |\wt{w}_4(\tau_4,n_4)|$ and $n' = n_1 - n_2$. Then the left-hand side of \eqref{eq:bilinear-1.1} is bounded by
\[\int_{\R}\sum_{\substack{n_2,n_3 \\ n' \neq 0}}\frac{1}{|n'|^{1+\epsilon}}\wh{f}_1(t,n'- n_2)\ol{\wh{f}}_2(t,n_2)\wh{f}_3(t,n_3)\ol{\wh{g}}_4(-t,n' + n_3) \; dt.\]
We focus on the integrand with respect to the spatial frequencies. We observe the following calculation
\begin{equation}\label{eq:convolution}
\begin{aligned}
\sum_{\substack{n_2,n_3 \\ n' \neq 0}}&\frac{1}{|n'|^{1+\epsilon}}\wh{g}_1(n'- n_2)\ol{\wh{g}}_2(n_2)\wh{g}_3(n_3)\ol{\wh{g}}_4(n' + n_3) \\
&\lesssim \sum_{n' \neq 0}\frac{1}{|n'|^{1+\epsilon}}\ft_x[g_1\ft_x^{-1}[\ol{\wh{g}}_2]](n')\ft_x[g_3\ol{g}_4](-n')\\
&\lesssim \norm{\ft_x[g_1\ft_x^{-1}[\ol{\wh{g}}_2]]}_{\ell_n^{\infty}}\norm{\ft_x[g_3\ol{g}_4]}_{\ell_n^{\infty}} \lesssim \prod_{j=1}^{4}\norm{g_j}_{L_x^2}.
\end{aligned}
\end{equation}
The second inequality holds since $\frac{1}{|n'|^{1+\epsilon}}$ is summable over $n' \neq 0$. Then, by using \eqref{eq:convolution} and the Sobolev embedding \eqref{eq:sobolev-1}, we obtain
\begin{equation}\label{eq:convolution-1}
\begin{aligned}
\mbox{LHS of } \eqref{eq:bilinear-1.1} &\lesssim \int_{\R} F_1(t)F_2(t)F_3(t)G_4(t) \; dt \lesssim \prod_{j=1}^{3}\norm{F_j}_{L_t^6}\norm{G_4}_{L_t^2}\\
&\lesssim \prod_{j=1}^{3}\norm{f_j}_{X^{0,\frac13}}\norm{w_4}_{L_{t,x}^2} \lesssim \prod_{j=1}^{3}\norm{u_j}_{X^{0,-\frac16}},
\end{aligned}
\end{equation}
where $F_j(t) = \norm{f_j(t)}_{L_x^2}$, $j=1,2,3$ and $G_4(t) = \norm{w(t)}_{L_x^2}$.

Now we prove \eqref{eq:bilinear-2}. Similarly as before, it suffices to show that
\begin{equation}\label{eq:bilinear-2.1}
\begin{aligned}
\Bigg\|\sum_{\substack{n_4, \N_{n_4}\\ |n_i|\sim N_i}}\int\limits_{\tau_1 - \tau_2 + \tau_3 = \tau_4} \frac{N_{max}^{\frac12 - \epsilon}}{\bra{\la_4}\prod_{j=1}^{3} \bra{\la_j}^{\frac12}}\wt{u}_1(\tau_1,n_1)\ol{\wt{u}}_2(\tau_2,n_2)\wt{u}_3(\tau_3,n_3)&\; d\tau_1d\tau_2\Bigg\|_{\ell_{n_4}^2L_{\tau_4}^1} \\ 
&\lesssim \prod_{j=1}^{3} \norm{u_j}_{L_{t,x}^2},
\end{aligned}
\end{equation}
where $0 < \epsilon < 1/2$. Without loss of generality we may assume (by the same reason as before) that $|\la_1| \ge |\la_2| \ge |\la_3|$, and we estimate the left-hand side of \eqref{eq:bilinear-2.1} by dividing into two cases: $|\la_1| \ge |\la_4|$ and $|\la_1| \le |\la_4|$.

For the case when $|\la_1| \ge |\la_4|$, since $\bra{\la_4}^{-\frac23}$ is $L_{\tau_4}^2$-integrable, we use the Cauchy-Schwarz inequality with respect to $\tau_4$, $|\la_1| \gtrsim |n_1-n_2|(n^{\ast})^2$ and the duality to dominate the left-hand side of \eqref{eq:bilinear-2.1} by
\begin{equation}\label{eq:bilinear-2.2}
\int_{\R}\sum_{\substack{n_1, n_2, n_3 \\ n_1 - n_2 \neq 0}}\frac{1}{|n_1 - n_2|^{1+\epsilon}}\wh{f}_1(t,n_1)\ol{\wh{f}}_2(t,n_2)\wh{f}_3(t,n_3)\ol{\wh{g}}_4(-t,n_4) \; dt,
\end{equation}
where $\wt{f}_1(\tau_1,n_1) = |\wt{u}_1(\tau_1,n_1)|$, $\wt{f}_j(\tau_j,n_j) = \bra{\tau_j - n_j^4}^{-\frac12}|\wt{u}_j(\tau_j,n_j)|$, $j=2,3$ and $\wt{g}_4(\tau_4,n_4) = \bra{\tau_4 - n_4^4}^{-\frac13}|\wt{w}_4(\tau_4,n_4)|$ with $\norm{w_4}_{L_{t,x}^2} \le 1$.

Similarly as in \eqref{eq:convolution} and \eqref{eq:convolution-1}, we can have
\[\begin{aligned}
\eqref{eq:bilinear-2.2} &\lesssim \int_{\R} F_1(t)F_2(t)F_3(t)G_4(t) \; dt \lesssim \norm{F_1}_{L_t^2}\prod_{j=2}^{3}\norm{F_j}_{L_t^6}\norm{G_4}_{L_t^6}\\
&\lesssim \norm{u_1}_{L_{t,x}^2}\prod_{j=2}^{3}\norm{f_j}_{X^{0,\frac13}}\norm{g_4}_{X^{0,\frac13}} \lesssim \norm{u_1}_{X^{0,0}}\prod_{j=2}^{3}\norm{u_j}_{X^{0,-\frac16}},
\end{aligned}\]
where $F_j(t) = \norm{f_j(t)}_{L_x^2}$, $j=1,2,3$ and $G_4(t) = \norm{g_4(t)}_{L_x^2}$.

For the case when $|\la_4| \ge |\la_1|$, since $\epsilon > 0$, we can choose $\delta, \gamma > 0$ small enough such that $\gamma > \delta$ and $2\delta + \gamma < \epsilon$. Then, since $\bra{\la_4}^{-\frac12 - \delta}$ is $L_{\tau_4}^2$-integrable, we use the Cauchy-Schwarz inequality with respect to $\tau_4$, $|\la_4| \gtrsim |n_1-n_2|(n^{\ast})^2$ and the duality to dominate the left-hand side of \eqref{eq:bilinear-2.1} by
\begin{equation}\label{eq:bilinear-2.3}
\int_{\R}\sum_{\substack{n_1, n_2, n_3 \\ n_1 - n_2 \neq 0}}\frac{1}{|n_1 - n_2|^{1+\gamma - \delta}N_{max}^{\epsilon - 2\delta - \gamma}}\wh{f}_1(t,n_1)\ol{\wh{f}}_2(t,n_2)\wh{f}_3(t,n_3)\ol{\wh{g}}_4(-t,n_4) \; dt,
\end{equation}
where $\wt{f}_j(\tau_j,n_j) = \bra{\tau_j - n_j^4}^{-\frac12}|\wt{u}_j(\tau_j,n_j)|$, $j=1,2,3$ and $\wt{g}_4(\tau_4,n_4) = |\wt{w}_4(\tau_4,n_4)|$ with $\norm{w_4}_{L_{t,x}^2} \le 1$. Since $\frac{1}{N_{max}^{\epsilon - 2\delta - \gamma}} \le 1$ and $\gamma - \delta >0$, by using similar arguments in \eqref{eq:convolution} and \eqref{eq:convolution-1}, we obtain
\[\begin{aligned}
\eqref{eq:bilinear-2.3} &\lesssim \int_{\R} F_1(t)F_2(t)F_3(t)G_4(t) \; dt \lesssim \prod_{j=1}^{3}\norm{F_j}_{L_t^6}\norm{G_4}_{L_t^2}\\
&\lesssim \prod_{j=2}^{3}\norm{f_j}_{X^{0,\frac13}}\norm{w_4}_{L_{t,x}^2} \lesssim \prod_{j=1}^{3}\norm{u_j}_{X^{0,-\frac16}},
\end{aligned}\]
where $F_j(t) = \norm{f_j(t)}_{L_x^2}$, $j=1,2,3$ and $G_4(t) = \norm{w_4(t)}_{L_x^2}$. Therefore, we complete the proof.
\end{proof}

\begin{lemma}\label{lem:resonant}
Let $s \ge 0$. Then we have
\begin{equation}\label{eq:resonant}
\norm{\N_R(u,v,w)}_{Z^{s,-\frac12}} \lesssim \norm{u}_{Z^{s,\frac12}}\norm{v}_{Z^{s,\frac12}}\norm{w}_{Z^{s,\frac12}}.
\end{equation}
\end{lemma}

\begin{proof}
We first estimate the $X^{s,b}$ portion. For the term $\wh{u}(n)\ol{\wh{v}}(n)\wh{w}(n)$, since 
\begin{equation}\label{eq:reson1}
\norm{\bra{n}^{s}\wh{u}(n)\ol{\wh{v}}(n)\wh{w}(n)}_{\ell_n^2} \lesssim \norm{u}_{H^s}\norm{v}_{H^s}\norm{w}_{H^s},
\end{equation}
we have from the Sobolev embedding \eqref{eq:sobolev-1} that
\begin{equation}\label{eq:reson2}
\begin{aligned}
\norm{\ft_x^{-1}[\wh{u}(n)\ol{\wh{v}}(n)\wh{w}(n)]}_{X^{s,-\frac12}} &\lesssim \norm{\ft_x^{-1}[\bra{n}^{s}\wh{u}(n)\ol{\wh{v}}(n)\wh{w}(n)]}_{X^{0,0}}\\
&\lesssim \norm{\norm{u(t)}_{H^s}\norm{v}_{H^s}\norm{w}_{H^s}}_{L_t^2}\\
&\lesssim \norm{u}_{L_t^6(H^s)}\norm{v}_{L_t^6(H^s)}\norm{w}_{L_t^6(H^s)}\\ 
&\lesssim \norm{u}_{X^{s,\frac12}}\norm{w}_{X^{s,\frac12}}\norm{w}_{X^{s,\frac12}}.
\end{aligned}
\end{equation}
For the term $\left(\sum_{n_1}\wh{u}(n_1)\ol{\wh{v}}(n_1)\right)\wh{w}(n)$, since the following also holds
\begin{equation}\label{eq:reson3}
\norm{\bra{n}^{s}\Big(\sum_{n_1}\wh{u}(n_1)\ol{\wh{v}}(n_1)\Big)\wh{w}(n)}_{\ell_n^2} \lesssim \norm{u}_{H^s}\norm{v}_{H^s}\norm{w}_{H^s},
\end{equation}
we have similarly as before that
\begin{equation}\label{eq:reson4}
\begin{aligned}
\norm{\ft_x^{-1}\Big[\Big(\sum_{n_1}\wh{u}(n_1)\ol{\wh{v}}(n_1)\Big)\wh{w}(n)\Big]}_{X^{s,-\frac12}} &\lesssim \norm{\ft_x^{-1}\Big[\bra{n}^s\Big(\sum_{n_1}\wh{u}(n_1)\ol{\wh{v}}(n_1)\Big)\wh{w}(n)\Big]}_{X^{0,0}}\\
&\lesssim \norm{\norm{u(t)}_{H^s}\norm{v}_{H^s}\norm{w}_{H^s}}_{L_t^2}\\
&\lesssim \norm{u}_{L_t^6(H^s)}\norm{v}_{L_t^6(H^s)}\norm{w}_{L_t^6(H^s)}\\ 
&\lesssim \norm{u}_{X^{s,\frac12}}\norm{w}_{X^{s,\frac12}}\norm{w}_{X^{s,\frac12}}.
\end{aligned}
\end{equation}

Now we estimate the $\ell_n^2L_{\tau}^1$ portion. By using \eqref{eq:reson1} and \eqref{eq:reson3}, we have similarly as before that

\begin{equation}\label{eq:reson5}
\begin{aligned}
\norm{\bra{\tau - n^4}^{-1}\ft_t[\bra{s}^s\wh{u}(n)\ol{\wh{v}}(n)\wh{w}(n)]}_{\ell_n^2L_{\tau}^1} &\lesssim \norm{\ft_t[\bra{s}^s\wh{u}(n)\ol{\wh{v}}(n)\wh{w}(n)]}_{\ell_n^2L_{\tau}^2} \\
&\lesssim \norm{\norm{u(t)}_{H^s}\norm{v(t)}_{H^s}\norm{w(t)}_{H^s}}_{L_t^2}\\
&\lesssim \norm{u}_{L_t^6(H^s)}\norm{v}_{L_t^6(H^s)}\norm{w}_{L_t^6(H^s)} \\
&\lesssim \norm{u}_{X^{s,\frac12}}\norm{w}_{X^{s,\frac12}}\norm{w}_{X^{s,\frac12}}
\end{aligned}
\end{equation}
and
\begin{equation}\label{eq:reson6}
\begin{aligned}
\norm{\bra{\tau - n^4}^{-1}\ft_t\Big[\bra{s}^s\Big(\sum_{n_1}\wh{u}(n_1)\ol{\wh{v}}(n_1)\Big)\wh{w}(n)\Big]}_{\ell_n^2L_{\tau}^1} &\lesssim \norm{\ft_t\Big[\bra{s}^s\Big(\sum_{n_1}\wh{u}(n_1)\ol{\wh{v}}(n_1)\Big)\wh{w}(n)\Big]}_{\ell_n^2L_{\tau}^2} \\
&\lesssim \norm{\norm{u(t)}_{H^s}\norm{v(t)}_{H^s}\norm{w(t)}_{H^s}}_{L_t^2}\\
&\lesssim \norm{u}_{L_t^6(H^s)}\norm{v}_{L_t^6(H^s)}\norm{w}_{L_t^6(H^s)} \\
&\lesssim \norm{u}_{X^{s,\frac12}}\norm{w}_{X^{s,\frac12}}\norm{w}_{X^{s,\frac12}}.
\end{aligned}
\end{equation}

By gathering \eqref{eq:reson2}, \eqref{eq:reson4}, \eqref{eq:reson5} and \eqref{eq:reson6}, we complete the proof of \eqref{eq:resonant}.
\end{proof}

\section{Local well-posedness of 4WNLS below $L^2(\T)$}\label{sec:main}

In this section we prove Theorem \ref{thm:main}. First we recall the equation for 4WNLS in terms of the Fourier coefficients:
\begin{align}\label{eq:WNLS4}
\pt \wh{u}(n) -i\mu(n) \wh{u}(n) =& i  (|\wh{u}(n)|^2 - |\wh{u}_0(n)|^2)\wh{u}(n) -i \sum_{\N_n}\wh{u}(n_1)\owh{u}(n_2)\wh{u}(n_3), \quad \forall\ n\in \Z,\\
:=&\widehat{I(u)}(n)+\widehat{II(u)}(n), \nonumber
\end{align}
with $\mu(n)$ given by \eqref{eq:modified linear operator}.

\subsection{Existence}

Following the strategy explained in the introduction, we will use $Y^{s,1/2}$ to study the equation \eqref{eq:WNLS4}. The standard $X^{s,b}$ analysis gives
\begin{equation}\label{eq:linear}
\norm{u}_{Y^{s,1/2}_T} \lesssim \norm{u_0}_{H^s}+\norm{I(u)}_{L_T^2H^s}+\norm{II(u)}_{Y_T^{s,-1/2+\epsilon}}
\end{equation}
for $\epsilon>0$. The second term $II(u)$ is non-resonant and thus easy to handle. Indeed, we have
\begin{proposition}\label{prop:nonres-trilinear} 
Let $-1/2 < s < 0$, $0 < T \le 1$, $t \in [-T,T]$ and $u_0 \in C^{\infty}(\T)$. Suppose that $u$ is a complex-valued smooth solution to \eqref{eq:WNLS4} on $[-T,T]$ and $u \in Y_T^{s,1/2}$. Then for $\delta = (s + 1/2)/3$, the following estimate holds:
\[\norm{II(u)}_{Y_T^{s,-1/2+\delta}}\lesssim \norm{u}_{Y_T^{s,1/2}}^3.\]
\end{proposition}
\begin{proof}
From the duality argument in addition to the Plancherel theorem, we know
\[\begin{aligned}
\norm{II(u)}_{Y^{s,-\frac12+ \delta}} &= \sup_{\norm{h}_{Y^{0,\frac12 - \delta}} \le 1} \Bigg|\sum_{n \in \Z} \bra{n}^{s}\int_0^T  \sum_{\N_n} \wh{u}(t,n_1)\owh{u}(t,n_2)\wh{u}(t,n_3)\owh{h}(t,n) \; dt \Bigg|\\
&= \sup_{\norm{h}_{Y^{s,\frac12-\delta}} \le 1} \Bigg|\sum_{n \in \Z} \bra{n}^{2s} \int_0^T \wh{u}(t,n_1)\owh{u}(t,n_2)\wh{u}(t,n_3)\owh{h}(t,n) \; dt \; ds \Bigg|.
\end{aligned}\]
Hence, it suffices to show 
\begin{equation}\label{eq:trilinear}
\sum_{n \in \Z} \Bigg|\bra{n}^{2s}\mbox{Im}\Bigg[ \int_0^t  \sum_{\N_n} \wh{u}(s,n_1)\owh{u}(s,n_2)\wh{u}(s,n_3)\owh{h}(s,n) \; ds \Bigg]\Bigg|\lesssim \norm{u}_{Y_T^{s,\frac12}}^3\norm{h}_{Y_T^{s,\frac12-\delta}}.
\end{equation}
We first note from the identities 
\[n_1 - n_2 + n_3 = n \quad \mbox{and} \quad \tau_1 - \tau_2 + \tau_3 = \tau\]
that
\begin{equation}\label{eq:resonant relation2}
\begin{aligned}
G&:= \left(\tau_1 - \mu(n_1)\right) - \left(\tau_2 - \mu(n_2)\right) +\left(\tau_3 - \mu(n_3)\right)-\left(\tau - \mu(n)\right)\\
&=(n_1-n_2)(n_2-n_3)\left(n_1^2 + n_2^2 + n_3^2 + (n_1-n_2+n_3)^2 +2(n_1+n_3)^2 \right)\\
&+\left(|\wh{u}_0(n_1)|^2 - |\wh{u}_0(n_2)|^2 + |\wh{u}_0(n_3)|^2 - |\wh{u}_0(n)|^2\right)
\end{aligned}
\end{equation}
and
\begin{equation}\label{eq:modulation0}
\max\left\{|\tau - \mu(n)|, |\tau_j - \mu(n_j)|:j=1,2,3 \right\} \gtrsim |G|.
\end{equation}

Let $n^{\ast} = \max(|n_1|,|n_2|,|n_3|,|n|)$. Then, from \eqref{eq:resonant relation2} in addition to Remark \ref{rem:negligible factor}, we know
\begin{equation}\label{eq:hl2-0}
|G| \gtrsim |n_1-n_2||n_2-n_3| (n^{\ast})^2.
\end{equation}
We can show \eqref{eq:trilinear} by dividing several cases as follows:
\begin{itemize}
\item high $\times$ high $\times$ high $\Rightarrow$ high,
\item high $\times$ high $\times$ high $\Rightarrow$ low,
\item high $\times$ high $\times$ low $\Rightarrow$ high,
\item high $\times$ high $\times$ low $\Rightarrow$ low,
\item high $\times$ low $\times$ low $\Rightarrow$ high.
\end{itemize}

We may assume from \eqref{eq:modulation0} and\eqref{eq:hl2-0} that $|\tau - \mu(n)| \gtrsim |n_1-n_2||n_2-n_3| (n^{\ast})^2$ without loss of generality.\footnote{In other cases, it is enough to switch roles of $\wh{h}(n)$ and one of $\wh{u}(n_1)$, $\wh{u}(n_2)$ and $\wh{u}(n_3)$.} We also assume that $|\tau_1-\mu(n_1)|$ is the second maximum modulation.\footnote{In view of \eqref{eq:low result1} below, the choice of the second modulation does not affect our analysis. Thus, we do not further consider the case when one of $|\tau_2-\mu(n_2)|$ and $|\tau_3-\mu(n_3)|$ is the second maximum modulation.} Let $\wt{f}_1(\tau_1,n_1) = \bra{\tau_1 - \mu(n_1)}^{\epsilon}\bra{n_1}^s|\wt{u}(\tau_1,n_1)|$, $\wh{f}_2(n_2) = \bra{n_2}^s|\wh{u}(n_2)|$, $\wt{f}_3(\tau_3,n_3) = \bra{\tau_3 - \mu(n_3)}^{-\epsilon}\bra{n_3}^s|\wt{u}(\tau_3,n_3)|$ and $\wt{g}(\tau,n) = \bra{\tau - \mu(n)}^{1/2-\delta}\bra{n}^s|\wt{h}(\tau,n)|$ for $0 < \epsilon \ll 1$. Then, the left-hand side of \eqref{eq:trilinear} is reduced by
\begin{equation}\label{eq:tri1}
\int_0^t \sum_{n,\N_n} \frac{\bra{n}^{s}\bra{n_1}^{-s}\bra{n_2}^{-s}\bra{n_3}^{-s}}{(|n_1-n_2||n_2-n_3|(n^{\ast})^2)^{1/2-\delta}}\wh{f}_1(n_1)\owh{f}_2(n_2)\wh{f}_3(n_3)\owh{g}(n) \; ds.
\end{equation}
Let denote the multiplier in the summand of \eqref{eq:tri1} by
\[m(n_1,n_2,n_3,n) := \frac{\bra{n}^{s}\bra{n_1}^{-s}\bra{n_2}^{-s}\bra{n_3}^{-s}}{(|n_1-n_2||n_2-n_3|(n^{\ast})^2)^{1/2-\delta}}.\]
Then, since $1 \le |n_1-n_2|, |n_2-n_3| \le n^{\ast}$, the multiplier $m(n_1,n_2,n_3,n)$ is roughly reduced as follows in each case provided above:
\begin{itemize}
\item high $\times$ high $\times$ high $\Rightarrow$ high 
\[m(n_1,n_2,n_3,n) \lesssim \frac{1}{|n_1-n_2|^{\frac32 + 2s - 3\delta}},\]
\item high $\times$ high $\times$ high $\Rightarrow$ low
\[m(n_1,n_2,n_3,n) \lesssim \frac{1}{|n_1-n_2|^{2+3s-4\delta}},\]
\item other cases
\[m(n_1,n_2,n_3,n) \lesssim \frac{1}{|n_1-n_2|^{2+2s-4\delta}}.\]
\end{itemize}
We note that for fixed $-1/2 < s < 0$, we easily see that
\[s-\delta > -\frac12 \hspace{1em} \mbox{and} \hspace{1em} s - \Big(\frac{s+1/2}{2} \Big) > -1/2,\]
where $\delta = (s+1/2)/3$. Since $3/2 + 2s - 3\delta < 2 + 3s - 4\delta < 2 + 2s - 4\delta$ for $-1/2 < s < 0$, it is enough to consider 
\[m(n_1,n_2,n_3,n) \lesssim \frac{1}{|n_1-n_2|^{\frac32 + 2s-3\delta}}.\]
Then, we have from $L^4$ Strichartz estimate \eqref{eq:L4} and the embedding theorem ($Y^{s,b} \hookrightarrow C_tH^s$, $b>1/2$) that
\begin{equation}\label{eq:low result1}
\begin{aligned}
\mbox{LHS of }\eqref{eq:trilinear} &\lesssim \int_0^t \sum_{n,\N_n} \frac{1}{|n_1-n_2|^{3/2+2s-3\delta}}\wh{f}_1(n_1)\owh{f}_2(n_2)\wh{f}_3(n_3)\owh{g}(n) \; ds\\
&\le \int_0^t \sum_{n' \neq 0}  \frac{1}{|n'|^{3/2+2s-3\delta}} \wh{f_1 \overline{f}_2}(n')\wh{f_3 \overline{g}}(-n') \; ds\\
&\lesssim \norm{f_1}_{L_{t,x}^4}\norm{f_2}_{L_{t,x}^4}\norm{f_3}_{L_t^{\infty}L_x^2}\norm{g}_{L_{t,x}^2}\\
&\lesssim \norm{u}_{Y^{s,\frac12}}^3\norm{h}_{Y^{s,\frac12-\delta}},
\end{aligned}
\end{equation}
for $0 < \epsilon < 1/4$. The third inequality in \eqref{eq:low result1} holds thanks to 
\[3/2 + 2s -3\delta = 3/2 + 2\left( s - \frac{s+1/2}{2} \right) > 1/2.\]
\end{proof}

\begin{remark}\label{rem:Z}
By the same proof above, we can actually prove: for $s>-1/2$
\begin{align*}
\norm{II(u)}_{Z_T^{s,-1/2}}\lesssim \norm{u}_{Z_T^{s,1/2}}^3.
\end{align*}
Thus the equation \eqref{eq:WNLS1}, if with only non-resonant term, is locally well-posed in $H^s$ for $s>-1/2$.
\end{remark}

As an immediate result of Proposition \ref{prop:nonres-trilinear}, we have the following corollary:
\begin{corollary}\label{cor:trilinear}
Let $-1/2 < s < 0$, $0 < T \le 1$, $t \in [-T,T]$ and $u_0 \in C^{\infty}(\T)$. Suppose that $u$ is a complex-valued smooth solution to \eqref{eq:WNLS4} on $[-T,T]$ and $u \in Y_T^{s,1/2}$. Then $u$ satisfies
\begin{equation}\label{eq:trilinear1}
\sum_{n \in\Z} \bra{n}^{2s}|\wh{u}(t,n)|^2\lesssim \norm{u_0}_{H^{s}}^2 + \norm{u}_{Y_T^{s,\frac12}}^4.
\end{equation}
\end{corollary}

\begin{proof}
From \eqref{eq:WNLS4}, we have
\begin{equation}\label{eq:energy}
\pt |\wh{u}(t,n)|^2 = -2\mbox{Im}\Big[\sum_{\N_n} \wh{u}(t,n_1)\owh{u}(t,n_2)\wh{u}(t,n_3)\owh{u}(t,n)\Big].
\end{equation}
Then, multiplying $\bra{n}^{2s}$ and taking the summation and the integration with respect to $x$ and $t$, respectively, to the both side of \eqref{eq:energy} in addition to Proposition  \ref{prop:nonres-trilinear} yield that
\begin{equation}\label{eq:res-nonres}
\begin{aligned}
\sum_{n \in \Z}\bra{n}^{2s}\big||\wh{u}(t,n)|^2-|\wh{u}_0(n)|^2\big| &=2\sum_{n \in \Z} \Bigg|\bra{n}^{2s}\mbox{Im}\Bigg[ \int_0^t  \sum_{\N_n} \wh{u}(s,n_1)\owh{u}(s,n_2)\wh{u}(s,n_3)\owh{u}(s,n) \; ds \Bigg]\Bigg|\\
&\lesssim \norm{u}_{Y_T^{s,\frac12}}^4.
\end{aligned}
\end{equation}
\end{proof}

From Proposition \ref{prop:nonres-trilinear} and Remark \ref{rem:Z}, we know the enemy to prevent low-regularity well-posedness is the resonant term $I(u)$. For $I(u)$, we have
\begin{equation}\label{eq:I}
\norm{I(u)}_{L_T^2H^s}\lesssim \left(\sup_{t,n}\left||\wh{u}(n)|^2 - |\wh{u}_0(n)|^2\right|\right)\norm{u}_{Y^{s,1/2}_{T}}.
\end{equation}
The following proposition is the important ingredient to control the term $I$, in particular 
\begin{equation}\label{eq:gap}
\sup_{n}\left||\wh{u}(n)|^2 - |\wh{u}_0(n)|^2\right|,
\end{equation}
and with this, we can show the existence of a solution to \eqref{eq:WNLS4} below $L^2(\T)$:

\begin{proposition}\label{prop:main}
Let $-1/3 \le s < 0$, $0 < T \le 1$, $t \in [-T,T]$ and $u_0 \in C^{\infty}(\T)$. Suppose that $u$ is a complex-valued smooth solution to \eqref{eq:WNLS4} on $[-T,T]$ and $u \in Y_T^{s,1/2}$. Then the following estimate holds:
\begin{equation}\label{eq:main}
\begin{aligned}
\sup_{n \in \Z} \Bigg|\mbox{Im}\Bigg[ \int_0^t  \sum_{\N_n} \wh{u}(s,n_1)&\owh{u}(s,n_2)\wh{u}(s,n_3)\owh{u}(s,n) \; ds \Bigg]\Bigg|\\
&\lesssim \norm{u_0}_{H^s}^4 + \big(\norm{u_0}_{H^{s}}^2 + \norm{u}_{Y_T^{s,\frac12}}^4\big)^2 +  \norm{u}_{Y_T^{s,\frac12}}^4 + \norm{u}_{Y_T^{s,\frac12}}^6.
\end{aligned}
\end{equation}
\end{proposition}

\begin{proof}
Let us define the projection operator as follows: For $N = 2^{\Z_{\ge0}}$, let 
\begin{equation}\label{eq:I_N}
I_1 = [-1,1] \qquad \mbox{and} \qquad I_N = [-2N,-N/2] \cup [N/2, 2N], \quad N \ge 2.
\end{equation}
We define $P_N$ by
\[\ft_x[P_Nf](n) = \chi_{I_N}(n)\wt{f}(n),\]
where $\chi_E$ is the characteristic function on $E$. We use the convention
\[P_{\le N} = \sum_{M \le N} P_M, \hspace{1em} P_{>N} = \sum_{M > N} P_M.\]
Then, the left-hand side of \eqref{eq:main} bounded by
\begin{equation}\label{eq:main1}
\sup_{N\ge1}\sum_{n \in I_N} \Bigg|\mbox{Im}\Bigg[ \int_0^t  \sum_{\N_n} \wh{u}(s,n_1)\owh{u}(s,n_2)\wh{u}(s,n_3)\owh{u}(s,n) \; ds \Bigg]\Bigg|.
\end{equation}
For fixed $N$, we further decompose $u$ in the integrand above into the following three pieces:
\[u = u_{low} + u_{med} + u_{high},\]
where $u_{med} =  P_{N}u$, $u_{low} = P_{\le N}u -  u_{med}$ and $u_{high} =  P_{\ge N}u - u_{med}$. Then, \eqref{eq:main1} can be divided into several cases.

\textbf{Case I.} (\emph{high $\times$ high $\times$ high $\Rightarrow$ high}) It suffices to control the following term 
\begin{equation}\label{eq:hhhh0}
\sup_{N\ge1}\sum_{n \in I_N} \Bigg|\mbox{Im}\Bigg[ \int_0^t  \sum_{\N_n} \wh{u}_{med}(s,n_1)\owh{u}_{med}(s,n_2)\wh{u}_{med}(s,n_3)\owh{u}(s,n) \; ds \Bigg]\Bigg|.
\end{equation}
From the following observation
\[\begin{aligned}
\pt \left( e^{-itn^4}\wh{u}(n) \right) &= e^{-itn^4} \left( \pt \wh{u}(n) - in^4 \wh{u}(n) \right)\\
&= e^{-itn^4} \left( i|\wh{u}(n)|^2\wh{u}(n) - i\sum_{\N_n}\wh{u}(n_1)\owh{u}(n_2)\wh{u}(n_3) \right),
\end{aligned}\]
we can apply the integration by parts with respect to the time variable $s$ to get
\[\begin{aligned}
\int_0^t&  \sum_{\N_n} \wh{u}(s,n_1)\owh{u}(s,n_2)\wh{u}(s,n_3)\owh{u}(s,n) \; ds \\
&=	\int_0^t  \sum_{\N_n} e^{-is(n_1^4 - n_2^4 + n_3^4 - n^4)}\big(e^{-isn_1^4} \wh{u}(s,n_1)\big)\big(\ol{e^{-isn_2^4} \wh{u}(s,n_2)} \big) \big(e^{-isn_3^4} \wh{u}(s,n_3) \big) \big(\ol{e^{-isn^4} \wh{u}(s,n)} \big) \; ds\\
&=\sum_{\N_n} \frac{1}{iH}\left( \wh{u}(t,n_1)\owh{u}(t,n_2)\wh{u}(t,n_3)\owh{u}(t,n) - \wh{u}_0(n_1)\owh{u}_0(n_2)\wh{u}_0(n_3)\owh{u}_0(n) \right)\\
&-\int_0^t \sum_{\N_n} \frac{e^{is(n_1^4 - n_2^4 + n_3^3 - n^4)}}{iH} \cdot \frac{d}{ds}\Big[\big(e^{-isn_1^4} \wh{u}(s,n_1)\big)\big(\ol{e^{-isn_2^4} \wh{u}(s,n_2)} \big) \big(e^{-isn_3^4} \wh{u}(s,n_3) \big) \big(\ol{e^{-isn^4} \wh{u}(s,n)} \big) \Big]\; ds,
\end{aligned}\]
where $H$ is defined as in \eqref{eq:resonant function}. \eqref{eq:hhhh0} is reduced as follows:
\[\begin{aligned}
&\eqref{eq:hhhh0}\\
&\le \sup_{N\ge1}\sum_{n \in I_N} \left| \sum_{\N_n} \frac{1}{H}\left( \wh{u}_{med}(t,n_1)\owh{u}_{med}(t,n_2)\wh{u}_{med}(t,n_3)\owh{u}(t,n) - \wh{u}_{0,med}(n_1)\owh{u}_{0,med}(n_2)\wh{u}_{0,med}(n_3)\owh{u}_{0}(n) \right)\right|\\
&+\sup_{N\ge1}\sum_{n \in I_N} \Bigg|\int_0^t \sum_{\N_n} \frac{e^{is(n_1^4 - n_2^4 + n_3^3 - n^4)}}{H} \\
&\hspace{9em}\times \frac{d}{ds}\Big[\big(e^{-isn_1^4} \wh{u}_{med}(s,n_1)\big)\big(\ol{e^{-isn_2^4} \wh{u}_{med}(s,n_2)} \big) \big(e^{-isn_3^4} \wh{u}_{med}(s,n_3) \big) \big(\ol{e^{-isn^4} \wh{u}(s,n)} \big) \Big]\; ds \Bigg|\\
&=: I + II.
\end{aligned}\]

For $I$, it is enough to consider
\begin{equation}\label{eq:hhhh1}
\sup_{N\ge1}\sum_{n \in I_N} \left| \sum_{\N_{n_4}} \frac1H \wh{f}_{1,med}(n_1)\owh{f}_{2,med}(n_2)\wh{f}_{3,med}(n_2)\owh{f}_{4,med}(n_4) \right|
\end{equation}
Let $\wh{g}_i(n) = \bra{n}^s|\wh{f}_{i,med}(n)|$, $i =1,3$, $\wh{g}_i(n) = \bra{n}^s|\wh{\overline{f}}_{i,med}(-n)|$, $i=2,4$. Since $|H| \gtrsim |(n_1-n_2)(n_2-n_3)|n_4^2 $ and $|n_4| \sim N$, we have
\begin{equation}\label{eq:nonresonant estimate}
\begin{aligned}
\eqref{eq:hhhh1} &\lesssim \sup_{N\ge1}N^{-(2+4s)}\sum_{n_4, \N_{n_4}} \frac{1}{|n_1 -n_2|} \wh{g}_1(n_1)\wh{g}_2(-n_2)\wh{g}_3(n_3)\wh{g}_4(-n_4) \\
&\le \sup_{N\ge1}N^{-(2+4s)}\sum_{n_1,n_4, n' \neq 0} \frac{1}{|n'|}\wh{g}_1(n_1)\wh{g}_2(n'-n_1)\wh{g}_3(n_4-n')\wh{g}_4(-n_4)\\
&=\sup_{N\ge1}N^{-(2+4s)}\sum_{0 < |n'| \le N} \frac{1}{|n'|}\wh{g_1g_2}(n')\wh{g_3g_4}(-n')\\
&\lesssim \sup_{N\ge1}N^{-(2+4s)}\log N\norm{\wh{g_1g_2}}_{\ell^{\infty}}\norm{\wh{g_3g_4}}_{\ell^{\infty}}\\
&\lesssim \prod_{j=1}^{4}\norm{f_j}_{H^s},
\end{aligned}
\end{equation}
whenever $2 + 4s > 0  \Rightarrow -1/2 < s < 0$. Hence, from \eqref{eq:trilinear1}, we obtain
\[I \lesssim \norm{u_0}_{H^s}^4 + \norm{u(t)}_{H^s}^4 \lesssim \norm{u_0}_{H^s}^4 + \big(\norm{u_0}_{H^{s}}^2 + \norm{u}_{Y_T^{s,\frac12}}^4\big)^2,\]
whenever $-1/2 < s < 0$.

\begin{remark}\label{rem:nontrivial}
This procedure cannot be directly applied for the uniqueness part, since we do not have \eqref{eq:trilinear1} for the difference of two solutions. However, we use a trick in \eqref{eq:nonresonant estimate} to get the resonance estimate for the difference of two solutions. See Proposition \ref{prop:main2} below.
\end{remark}

For $II$, we consider the case when the time derivative is taken in the $n_1$-frequency mode. Then, $II$ is rewritten as  
\begin{equation}\label{eq:hhhh2}
\begin{aligned}
\sup_{N \ge 1}\sum_{n \in I_N} &\left| \int_0^t \sum_{\N_n} \frac1H \left[ |\wh{u}_{med}(n_1)|^2\wh{u}_{med}(n_1) - P_{N}\sum_{\N_{n_1}} \wh{u}(n_{11})\owh{u}(n_{12})\wh{u}(n_{13}) \right]\owh{u}_{med}(n_2)\wh{u}_{med}(n_3)\owh{u}(n) \; ds \right| \\
&=: II_1 + II_2.
\end{aligned}
\end{equation}
We remark that the estimate does not depend on the choice of functions in which the time derivative is taken, and hence it is enough to consider only this case above.

For the part $II_1$, we recall the resonant relation associated to $\mu(n)$
\begin{equation}\label{eq:resonant relation}
\begin{aligned}
G&:= \left(\tau_1 - \mu(n_1)\right) - \left(\tau_2 - \mu(n_2)\right) +\left(\tau_3 - \mu(n_3)\right)-\left(\tau - \mu(n)\right)\\
&=(n_1-n_2)(n_2-n_3)\left(n_1^2 + n_2^2 + n_3^2 + (n_1-n_2+n_3)^2 +2(n_1+n_3)^2 \right)\\
&+\left(|\wh{u}_0(n_1)|^2 - |\wh{u}_0(n_2)|^2 + |\wh{u}_0(n_3)|^2 - |\wh{u}_0(n)|^2\right)
\end{aligned}
\end{equation}
and the support property
\[\max\left\{|\tau - \mu(n)|, |\tau_j - \mu(n_j)|:j=1,2,3 \right\} \gtrsim |G|.\]
Similarly as the estimate of $I$, we first consider 
\begin{equation}\label{eq:endpoint1}
\sup_{N \ge 1}\sum_{n \in I_N,\N_n} \left| \frac1H |\wh{u}_{med}(n_1)|^2\wh{u}_{med}(n_1)\owh{u}_{med}(n_2)\wh{u}_{med}(n_3)\owh{u}(n) \right|.
\end{equation}
We assume that $|\tau - \mu(n)| \gtrsim |G|$. Let us define 
\[\wh{g}(n) = \bra{n}^s\wh{u}_{med}(n) \hspace{1em} \mbox{and} \hspace{1em} \wt{h}(n) = \bra{\tau - \mu(n)}^{\frac12}\bra{n}^s\wt{u}(\tau,n).\]
Since $|\wh{u}_{med}(n_1)|^2 \le \norm{u}_{H^s}^2$, the similar argument as in \eqref{eq:nonresonant estimate} yields
\begin{equation}\label{eq:endpoint2}
\begin{aligned}
\eqref{eq:endpoint1} &\lesssim \norm{u}_{H^s}^2\sup_{N \ge 1}N^{-(3+6s)}\sum_{n' \neq 0} \frac{1}{|n'|^{3/2}}\wh{g\overline{g}}(n')\wh{g\overline{h}}(-n')\\
&\lesssim \norm{u}_{H^s}^2\sup_{N \ge 1}N^{-(3+6s)}\norm{\wh{g\overline{g}}}_{\ell^{\infty}}\norm{\wh{g\overline{h}}}_{\ell^{\infty}}\\
&\lesssim \norm{u}_{H^s}^5 \norm{\ft^{-1}[\bra{\tau-\mu(n)}^{1/2}\bra{n}^s\wh{u}(n)]}_{L_x^2}
\end{aligned}
\end{equation}
whenever $3+6s\ge0 \Rightarrow -1/2 \le s < 0$. By the Sobolev embedding property (Lemma \ref{lem:sobolev}), we finally have
\[II_1 \lesssim \norm{u}_{L_t^{10}H^s}^5 \norm{\ft^{-1}[\bra{\tau-\mu(n)}^{1/2}\bra{n}^s\wh{u}(n)]}_{L_{t,x}^2} \lesssim \norm{u}_{Y^{s,\frac12}}^6.\]
We can see that the choice of the maximum modulation does not affect \eqref{eq:endpoint2}, and thus we do not need to dear with the other cases. 

For the part $II_2$, we further decompose $\wh{u}(n_{1,i})$ into $\wh{u}_{low}(n_{1,i}), \wh{u}_{med}(n_{1,i})$ and $\wh{u}_{high}(n_{1,i})$, $i=1,2,3$. Then, $II_2$ can be treated by dividing $\wh{u}(n_{11})\owh{u}(n_{12})\wh{u}(n_{13})$ into the following cases:
\[\wh{u}_{med}(n_{11})\owh{u}_{med}(n_{12})\wh{u}_{med}(n_{13}), \tag{Case A}\]
\[\wh{u}_{med}(n_{11})\owh{u}_{low}(n_{12})\wh{u}_{low}(n_{13}) \hspace{1em} (\Leftrightarrow \wh{u}_{low}(n_{11})\owh{u}_{low}(n_{12})\wh{u}_{med}(n_{13})), \tag{Case B-1}\]
\[\wh{u}_{low}(n_{11})\owh{u}_{med}(n_{12})\wh{u}_{low}(n_{13}),\tag{Case B-2}\]
\[\wh{u}_{med}(n_{11})\owh{u}_{high}(n_{12})\wh{u}_{high}(n_{13}) \hspace{1em} (\Leftrightarrow \wh{u}_{high}(n_{11})\owh{u}_{med}(n_{12})\wh{u}_{high}(n_{13})), \tag{Case B-3}\]
\[\wh{u}_{high}(n_{11})\owh{u}_{med}(n_{12})\wh{u}_{high}(n_{13}), \tag{Case B-4}\]
\[\wh{u}_{low}(n_{11})\owh{u}_{high}(n_{12})\wh{u}_{high}(n_{13}) \hspace{1em} (\Leftrightarrow \wh{u}_{high}(n_{11})\owh{u}_{high}(n_{12})\wh{u}_{low}(n_{13})), \tag{Case B-5}\]
\[\wh{u}_{high}(n_{11})\owh{u}_{low}(n_{12})\wh{u}_{high}(n_{13}),\tag{Case B-6}\]
\[\wh{u}_{med}(n_{11})\owh{u}_{med}(n_{12})\wh{u}_{low}(n_{13}) \hspace{1em} (\Leftrightarrow \wh{u}_{low}(n_{11})\owh{u}_{med}(n_{12})\wh{u}_{med}(n_{13})), \tag{Case C-1}\]
\[\wh{u}_{med}(n_{11})\owh{u}_{low}(n_{12})\wh{u}_{med}(n_{13}),\tag{Case C-2}\]
\[\wh{u}_{high}(n_{11})\owh{u}_{high}(n_{12})\wh{u}_{high}(n_{13}).\tag{Case C-3}\]

\textbf{Case A} In this case, since all frequencies are comparable, we may not use the maximum modulation effect in view of the new resonant function defined below \eqref{eq:resonant relation1}: From the identities
\[n_{11} - n_{12} + n_{13} - n_2 + n_3 = n\]
and
\[\tau_{11} - \tau_{12} + \tau_{13} - \tau_2 + \tau_3 = \tau,\]
we know
\[\sum_{j=1}^3(-1)^{j-1}(\tau_{1j} - \mu(n_{1j})) - (\tau_{2} - \mu(n_2)) + \tau_{3} - \mu(n_3) = \tau - \mu(n) -\wt{G},\]
where $\wt{G}$ is defined as
\begin{equation}\label{eq:resonant relation1} 
\begin{aligned}
\wt{G}&:= \left(\tau_{11} - \mu(n_{11})\right) - \left(\tau_{12} - \mu(n_{12})\right) +\left(\tau_{13} - \mu(n_{13})\right)-\left(\tau_2 - \mu(n_2)\right) + \left(\tau_3 - \mu(n_3)\right) - \left(\tau - \mu(n)\right)\\
&=(n_1-n_2)(n_2-n_3)\left(n_1^2 + n_2^2 + n_3^2 + n^2 +2(n_1+n_3)^2 \right)\\ 
&+ (n_{11}-n_{12})(n_{12}-n_{13})\left(n_{11}^2 + n_{12}^2 + n_{13}^2 + n_1^2 +2(n_{11}+n_{13})^2 \right)\\
&+\left(|\wh{u}_0(n_{11})|^2 - |\wh{u}_0(n_{12})|^2 + |\wh{u}_0(n_{13})|^2 - |\wh{u}_0(n_2)|^2 + |\wh{u}_0(n_3)|^2 - |\wh{u}_0(n)|^2\right).
\end{aligned}
\end{equation}
Note that 
\begin{equation}\label{eq:modulation1}
L_{max}:=\max(|\tau-\mu(n)|, |\tau_i - \mu(n_i)|; i=2,3, |\tau_{1j}-\mu(n_{1j})|; j=1,2,3) \gtrsim |\wt{G}|.
\end{equation}

We assume $L_{max} = |\tau - \mu(n)|$. Let us define 
\[\wh{g}(n) = \bra{n}^s\wh{u}_{med}(n),\]
\[\wt{h}(n) = \bra{\tau - \mu(n)}^{-\ep}\bra{n}^s\wt{u}_{med}(n)\]
and
\[\wt{f}(n) = \bra{\tau - \mu(n)}^{2\ep}\bra{n}^s\wt{u}_{med}(n),\]
where $\ep > 0$ will be chosen later. Since 
\[\bra{\tau - \mu(n)}^{-2\ep} \le \bra{\tau_2 - \mu(n_2)}^{-\ep}\bra{\tau_{11} - \mu(n_{11})}^{-\ep}\]
and $|H| \gtrsim |n_1-n_2||n_2-n_3|N^2$, by performing the change of variables 
\[\begin{matrix}
n' = n_{12} - n_{13} \\ 
n''=n_3-n = n_2-n_1
\end{matrix}
\hspace{0.5em}\Rightarrow\hspace{0.5em} 
\begin{matrix}
n_{11} = n_{11}, \hspace{0.5em} n_{12} = n_{12},\hspace{0.5em} n_{13}= n_{12} - n',\\ 
n_2 = n_{11} - n' + n'', \hspace{0.5em} n_3 = n_3, \hspace{0.5em} n = n_3 - n'',
\end{matrix}\] 
we have
\begin{equation}\label{eq:estimation-1}
\begin{aligned}
\sup_{N \ge 1}&\sum_{n \in I_N}\sum_{N_n}\frac{1}{H}P_N[\sum_{N_{n_1}}\wh{u}_{med}(n_{11})\owh{u}_{med}(n_{12})\wh{u}_{med}(n_{13})]\owh{u}_{med}(n_2)\wh{u}_{med}(n_3)\owh{u}(n)\\
&\lesssim \sup_{N \ge 1}N^{-2-6s}\sum_{n \in I_N}\sum_{N_n}\sum_{N_{n_1}}\frac{1}{|n_1-n_2||n_2-n_3|}\wh{h}(n_{11})\owh{g}(n_{12})\wh{g}(n_{13})\owh{h}(n_2)\wh{g}(n_3)\owh{f}(n)\\
&= \sup_{N \ge 1}N^{-2-6s}\sum_{\substack{n_{11},n_{12},n_3, n'\\|n_{11}-n'+n''-n_3| \neq 0\\ 0< |n''| \le N}}\frac{1}{|n''||n_{11}-n'+n''-n_3|}\\
&\hspace{11em}\times\wh{h}(n_{11})\owh{g}(n_{12})\wh{g}(n_{12}-n')\owh{h}(n_{11}-n'+n'')\wh{g}(n_3)\owh{f}(n_3-n'')\\
&\lesssim \sup_{N \ge 1}N^{-2-6s}\sum_{0 < |n''| \le N}\sum_{\substack{n_3,n_{11},n'\\|n_{11}-n'+n''-n_3| \neq 0}}\frac{1}{|n''||n_{11}-n'+n''-n_3|}\\
&\hspace{13em}\times\wh{h}(n_{11})\wh{\overline{g}g}(-n')\owh{h}(n_{11}-n'+n'')\wh{g}(n_3)\owh{f}(n_3-n'')\\
&\lesssim \sup_{N \ge 1}N^{-2-6s}\sum_{0 < |n''| \le N}\sum_{n_3,n_{11}}\frac{1}{|n''|}\bigg[\Big(\sum_{\substack{n'\\|n_{11}-n'+n''-n_3| \neq 0}}\frac{1}{|n_{11}-n'+n''-n_3|^2}\Big)^{1/2}\\
&\hspace{11em}\times\Big(\sum_{n'}|\wh{h}(n_{11})\wh{\overline{g}g}(-n')\owh{h}(n_{11}-n'+n'')\wh{g}(n_3)\owh{f}(n_3-n'')|^2\Big)^{1/2}\bigg]\\
&\lesssim \sup_{N \ge 1}N^{-2-6s}\sum_{0 < |n''| \le N}\sum_{n_3}\frac{1}{|n''|}|\wh{g}(n_3)\owh{f}(n_3-n'')|\\
&\hspace{11em}\times\bigg[\Big(\sum_{n_{11}}|\wh{h}(n_{11})|^2\Big)^{1/2}\Big(\sum_{n_{11},n'}|\wh{\overline{g}g}(-n')\owh{h}(n_{11}-n'+n'')|^2\Big)^{1/2}\bigg]\\
&\lesssim \sup_{N \ge 1}N^{-2-6s}\norm{h}_{L^2}^2\norm{\overline{g}g}_{L^2}\bigg[\sum_{0 < |n''| \le N}\Big(\sum_{n_3}\frac{1}{|n''|}|\wh{g}(n_3)\owh{f}(n_3-n'')|\Big)^2\bigg]^{1/2}\\
&\lesssim \sup_{N \ge 1}N^{-2-6s}\norm{h}_{L^2}^2\norm{\overline{g}g}_{L^2}\norm{\ft_x^{-1}[|\wh{g}|]\ft_x^{-1}[|\wh{\overline{f}}|]}_{L^2}\\
&\lesssim \norm{g}_{L^4}^2\norm{\ft_x^{-1}[|\wh{g}|]}_{L^4}\norm{\ft_x^{-1}[|\wh{\overline{f}}|]}_{L^4}\norm{h}_{L^2}^2,
\end{aligned}
\end{equation}
whenever $-1/3 \le s < 0$. Hence $L^4$ Strichartz estimate (Lemma \ref{lem:L4}) and Sobolev embedding ($Y^{s,\frac12+} \hookrightarrow L^{\infty}H^s$) guarantee
\begin{equation}\label{eq:estimation-3}
II_2 \lesssim \norm{g}_{L_{t,x}^4}^2\norm{\ft_x^{-1}[|\wh{g}|]}_{L_{t,x}^4}\norm{\ft_x^{-1}[|\wh{\overline{f}}|]}_{L_{t,x}^4}\norm{h}_{L^{\infty}L^2}^2 \lesssim \norm{u}_{Y^{s,\frac12}}^6,
\end{equation}
by choosing $0 < \ep < 3/32$.

In view of \eqref{eq:estimation-1}, we can see that the same result follows \eqref{eq:estimation-1} without any modification, if one of $|\tau_{12} - \mu(n_{13})|$, $|\tau_{12} - \mu(n_{13})|$ and $|\tau_{3} - \mu(n_{3})|$ is $L_{max}$. Besides, one can get the same result by slightly modifying the change of variables, if one of the rest modulations is $L_{max}$. Indeed, if $|\tau_{11} - \mu(n_{11})| = L_{max}$, we change the variable $n'=n_{12} - n_{11}$ instead of $n'=n_{12} - n_{13}$. Then, $\wh{u}_{med}(n_{11})$ and $\wh{u}_{med}(n_{13})$ switch their roles in \eqref{eq:estimation-1}, and hence we \eqref{eq:estimation-3}. Lastly, if $|\tau_{2} - \mu(n_{2})| = L_{max}$, by using the change of variable $n'' = n_3-n_2 = n-n_1$ instead of $n''=n_3-n = n_2-n_1$, we can switch the roles of $\wh{u}_{med}(n_2)$ and $\wh{u}(n)$, and hence \eqref{eq:estimation-3} follows \eqref{eq:estimation-1}. We would like to note that not the smoothing effect of the maximum modulation, but the spare room of modulations between $L^4$ and $Y^{s,\frac12}$ allows us to use this argument. 

\begin{remark}\label{rem:threshold}
Once we choose $\ell^1$-norm in \eqref{eq:gap} instead of $\ell^{\infty}$-norm, $\sup_{N \ge 1}$ should be replaced by $\sum_{N \ge 1}$ in \eqref{eq:estimation-1}, which implies \eqref{eq:estimation-3} holds for $s > -1/3$. Hence it can be known that the choice of $\ell^{\infty}$-norm  in \eqref{eq:gap} prevents the logarithmic divergence at the end point regularity $s = -1/3$.

On the other hand, this non-resonant contribution of the time derivative is the worst in the sense that the smoothing effect breaks down for $s < -1/3$. The reason follows exactly Remark 3.2 in \cite{NTT2010}.
\end{remark}

Now we consider the cases B and C. Let denote $\max(|n_{11}|, |n_{12}|, |n_{13}|)$ by $n^*$. From \eqref{eq:resonant relation1} and \eqref{eq:modulation1}, we know that
\[L_{max} \gtrsim |n_{11} - n_{12}||n_{12}-n_{13}|(n^*)^2.\]

\textbf{Case B-1} It suffices to consider 
\begin{equation}\label{eq:B1}
\sup_{N \ge 1}\sum_{n \in I_N}\sum_{N_n}\frac{1}{H}P_N[\sum_{N_{n_1}}\wh{u}_{med}(n_{11})\owh{u}_{low}(n_{12})\wh{u}_{low}(n_{13})]\owh{u}_{med}(n_2)\wh{u}_{med}(n_3)\owh{u}(n).
\end{equation}
We assume that $|\tau - \mu(n)| = L_{max}$. Let us define
\[\wh{g}(n) = \bra{n}^s\wh{u}(n)\]
and
\[\wt{f}(n) = \bra{\tau-\mu(n)}^{\frac12}\bra{n}^s\wt{u}(\tau, n).\]
Since $\bra{n_{1,j}}^{-s} \lesssim N^{-s}$, $j=1,2,3$, $|H| \gtrsim |n_1-n_2||n_2-n_3|N^2$,
\[L_{max}^{-\frac12} \lesssim (|n_{11} - n_{12}||n_{12}-n_{13}|N^2)^{-\frac12}\]
and $0 < |n_{12} - n_{13}| \le N$, the similar argument as \eqref{eq:estimation-1} gives us that
\begin{equation}\label{eq:estimation3}
\begin{aligned}
\eqref{eq:B1}&\lesssim \sup_{N \ge 1}N^{-3-6s}\sum_{n \in I_N}\sum_{N_n}\frac{1}{|n_1-n_2||n_2-n_3|}\\
&\hspace{7em}\times P_N[\sum_{N_{n_1}}\frac{1}{|n_{12}-n_{13}|^{1/2}}\wh{g}(n_{11})\owh{g}(n_{12})\wh{g}(n_{13})]\owh{g}(n_2)\wh{g}(n_3)\owh{f}(n)\\
&\lesssim \sup_{N \ge 1}N^{-3-6s}\sum_{0 < |n''| \le N}\sum_{n_3}\frac{1}{|n''|}|\wh{g}(n_3)\owh{f}(n''-n_3)|\\\
&\hspace{7em}\times\bigg[\Big(\sum_{n_{11}}|\wh{g}(n_{11})|^2\Big)^{1/2}\Big(\sum_{\substack{n_{11}\\0< |n'| \le N}}\frac{1}{|n'|}|\wh{\overline{g}g}(-n')\owh{g}(n_{11}-n'+n'')|^2\Big)^{1/2}\bigg]\\
&\lesssim \sup_{N \ge 1}N^{-3-6s}(\log N)^{\frac32} \norm{f}_{L^2}\norm{g}_{L^2}^3\norm{\overline{g}g}_{L^1}\\
&\lesssim \norm{g}_{L^2}^5\norm{f}_{L^2},
\end{aligned}
\end{equation}
for $-1/2 < s < 0$. Hence, by Sobolev embedding, we have
\[II_2 \lesssim \norm{g}_{L_t^{10}L_x^2}^5\norm{f}_{L_{t,x}^2} \lesssim \norm{u}_{Y^{s,\frac12}}^6.\]
We remark in view of \eqref{eq:estimation3} that we can obtain the same result without any modification, when $|\tau - \mu(n)| \neq L_{max}$.

\textbf{Case B-2} For the integrand
\[\sup_{N \ge 1}\sum_{n \in I_N}\sum_{N_n}\frac{1}{H}P_N[\sum_{N_{n_1}}\wh{u}_{low}(n_{11})\owh{u}_{med}(n_{12})\wh{u}_{low}(n_{13})]\owh{u}_{med}(n_2)\wh{u}_{med}(n_3)\owh{u}(n),\]
the same argument in \textbf{Case B-1} can be directly applied, so we have the same result.

\textbf{Case B-3} We consider the following integrand:
\begin{equation}\label{eq:B3}
\sup_{N \ge 1}\sum_{n \in I_N}\sum_{N_n}\frac{1}{H}P_N[\sum_{N_{n_1}}\wh{u}_{med}(n_{11})\owh{u}_{high}(n_{12})\wh{u}_{high}(n_{13})]\owh{u}_{med}(n_2)\wh{u}_{med}(n_3)\owh{u}(n).
\end{equation}
Given $-1/2 < s < 0$, let $\ep := (3+6s)/12$. Since 
\[L_{max}^{-\frac12} \le |n_{11}-n_{12}|^{-\frac12 - \ep}|n_{12}-n_{13}|^{-\frac12 - \ep}(n^*)^{-1+2\ep}\]
and
\[\bra{n_{11}}^{-s}\bra{n_{12}}^{-s}\bra{n_{13}}^{-s}(n^*)^{-1+2\ep} \le N^{-1-3s+2\ep},\]
we have similarly as in \eqref{eq:estimation-1} that
\[\begin{aligned}
\eqref{eq:B3}&\lesssim \sup_{N \ge 1}N^{-3-6s + 2\ep}\sum_{n \in I_N}\sum_{N_n}\frac{1}{|n_1-n_2||n_2-n_3|}\\
&\hspace{7em}\times P_N[\sum_{N_{n_1}}\frac{1}{|n_{12}-n_{13}|^{1/2+\ep}}\wh{g}(n_{11})\owh{g}(n_{12})\wh{g}(n_{13})]\owh{g}(n_2)\wh{g}(n_3)\owh{f}(n)\\
&\lesssim \sup_{N \ge 1}N^{-3-6s+2\ep}\sum_{0 < |n''| \le N}\sum_{n_3}\frac{1}{|n''|}|\wh{g}(n_3)\owh{f}(n''-n_3)|\\\
&\hspace{7em}\times\bigg[\Big(\sum_{n_{11}}|\wh{g}(n_{11})|^2\Big)^{1/2}\Big(\sum_{\substack{n_{11}\\n' \neq 0}}\frac{1}{|n'|^{1+2\ep}}|\wh{\overline{g}g}(-n')\owh{g}(n_{11}-n'+n'')|^2\Big)^{1/2}\bigg]\\
&\lesssim \sup_{N \ge 1}N^{-3-6s+2\ep}(\log N)\norm{f}_{L^2}\norm{g}_{L^2}^3\norm{\overline{g}g}_{L^1}\\
&\lesssim \norm{g}_{L^2}^5\norm{f}_{L^2},
\end{aligned}\]
which implies
\[II_2 \lesssim \norm{g}_{L_t^{10}L_x^2}^5\norm{f}_{L_{t,x}^2} \lesssim \norm{u}_{Y^{s,\frac12}}^6.\]

The cases \textbf{B-4}, \textbf{B-5}, \textbf{B-6} can be treated exactly same as \textbf{Case B-3} Moreover, the cases \textbf{C-1} and \textbf{C-2} can be treated by the same argument as in \textbf{Case B-1}, so we omit the details.

\textbf{Case C-3} We consider the following integrand:
\begin{equation}\label{eq:C3}
\sup_{N \ge 1}\sum_{n \in I_N}\sum_{N_n}\frac{1}{H}P_N[\sum_{N_{n_1}}\wh{u}_{high}(n_{11})\owh{u}_{high}(n_{12})\wh{u}_{high}(n_{13})]\owh{u}_{med}(n_2)\wh{u}_{med}(n_3)\owh{u}(n).
\end{equation}
In this case, we know that
\[L_{max}^{-\frac12} \le |n_{11}-n_{12}|^{-\frac12 - \ep}|n_{12}-n_{13}|^{-\frac12 - \ep}(n^*)^{-1+2\ep} \hspace{1em} \mbox{and} \hspace{1em} |n_{11}-n_{12}|, |n_{12}-n_{13}| \sim n^* \gtrsim N,\]
for $-1/2 < s < 0$ and $\ep := (3+6s)/4$. Since 
\[\bra{n_{11}}^{-s}\bra{n_{12}}^{-s}\bra{n_{13}}^{-s}(n^*)^{-3/2+\ep} \le N^{-3/2-3s+2\ep},\]
similarly as \textbf{Case B-3}, we have
\[\begin{aligned}
\eqref{eq:C3}&\lesssim \sup_{N \ge 1}N^{-7/2-6s + \ep}\sum_{n \in I_N}\sum_{N_n}\frac{1}{|n_1-n_2||n_2-n_3|}\\
&\hspace{7em}\times P_N[\sum_{N_{n_1}}\frac{1}{|n_{12}-n_{13}|^{1/2+\ep}}\wh{g}(n_{11})\owh{g}(n_{12})\wh{g}(n_{13})]\owh{g}(n_2)\wh{g}(n_3)\owh{f}(n)\\
&\lesssim \sup_{N \ge 1}N^{-7/2-6s + \ep}\sum_{0 < |n''| \le N}\sum_{n_3}\frac{1}{|n''|}|\wh{g}(n_3)\owh{f}(n''-n_3)|\\\
&\hspace{7em}\times\bigg[\Big(\sum_{n_{11}}|\wh{g}(n_{11})|^2\Big)^{1/2}\Big(\sum_{\substack{n_{11}\\|n'| \gtrsim N}}\frac{1}{|n'|^{1+2\ep}}|\wh{\overline{g}g}(-n')\owh{g}(n_{11}-n'+n'')|^2\Big)^{1/2}\bigg]\\
&\lesssim \sup_{N \ge 1}N^{-7/2-6s}(\log N)\norm{f}_{L^2}\norm{g}_{L^2}^3\norm{\overline{g}g}_{L^1}\\
&\lesssim \norm{g}_{L^2}^5\norm{f}_{L^2},
\end{aligned}\]
which implies
\[II_2 \lesssim \norm{g}_{L_t^{10}L_x^2}^5\norm{f}_{L_{t,x}^2} \lesssim \norm{u}_{Y^{s,\frac12}}^6.\]

Therefore, we have
\[\eqref{eq:hhhh0} \lesssim \norm{u_0}_{H^s}^4 + \big(\norm{u_0}_{H^{s}}^2 + \norm{u}_{Y_T^{s,\frac12}}^4\big)^2 + \norm{u}_{Y^{s,\frac12}}^6\]
for $-1/3 \le s <0$.

\textbf{Case II.} We deal with the interactions of two high frequencies and two low frequencies.

\textbf{Case II-a.} (\emph{high $\times$ low $\times$ low $\Rightarrow$ high}\footnote{\emph{low $\times$ low $\times$ high $\Rightarrow$ high} case can be estimated by the same argument due to the symmetry.}) It suffices to estimate
\begin{equation}\label{eq:hl0}
\sup_{N\ge1}\sum_{n \in I_N} \Bigg|\mbox{Im}\Bigg[ \int_0^t  \sum_{\N_n} \wh{u}_{med}(s,n_1)\owh{u}_{low}(s,n_2)\wh{u}_{low}(s,n_3)\owh{u}(s,n) \; ds \Bigg]\Bigg|.
\end{equation}
From \eqref{eq:resonant relation} in addition to Remark \ref{rem:negligible factor}, we know
\begin{equation}\label{eq:hl2}
|G| \gtrsim |n_2-n_3|n^3.
\end{equation}
We first assume $|\tau - \mu(n)| \gtrsim |n_2-n_3| n^3$. Given $-3/8 < s < 0$, we choose $\ep := (3/2 + 4s)/8 > 0$. Since 
\[\bra{\tau - \mu(n)}^{-\frac12} \le |G|^{-\frac12 + \ep} \bra{\tau_1 - \mu(n_1)}^{-\ep} \lesssim |n_2-n_3|^{1/2}|n|^{-\frac32 + 4\ep} \bra{\tau_1 - \mu(n_1)}^{-\ep},\]
for $\wt{f}_1(\tau_1,n_1) = \bra{\tau_1 - \mu(n_1)}^{-\epsilon}\bra{n_1}^s|\wt{u}_{med}(\tau_1,n_1)|$, $\wh{f}_2(n_2) = \bra{n_2}^s|\wh{u}_{low}(n_2)|$, $\wt{f}_3(\tau_3,n_3) = \bra{n_3}^s|\wt{u}_{low}(\tau_3,n_3)|$ and $\wt{g}(\tau,n) = \bra{\tau - \mu(n)}^{1/2}\bra{n}^s|\wt{u}(\tau,n)|$, we have 
\begin{equation}\label{eq:low result}
\begin{aligned}
\eqref{eq:hl0} &\lesssim \sup_{N\ge1}N^{-(3/2 + 4s -4\ep)}\int_0^t \sum_{n\in I_N,\N_n} \frac{1}{|n_2-n_3|^{1/2}}\wh{f}_1(n_1)\owh{f}_2(n_2)\wh{f}_3(n_3)\owh{g}(n) \; ds\\
&\le \sup_{N\ge1}N^{-(3/2 + 4s -4\ep)}\sum_{0 < |n'| \le N}  \frac{1}{|n'|^{1/2}} \wh{f_1 \overline{g}}(n')\wh{ \overline{f}_2f_3}(-n') \; ds\\
&\lesssim \sup_{N\ge1}N^{-(3/2+4s-4\ep)}\log N\norm{f_2}_{L_{t,x}^4}\norm{f_3}_{L_{t,x}^4}\norm{f_1}_{L_t^{\infty}L^2}\norm{g}_{L_{t,x}^2}\\
&\lesssim \norm{u}_{Y^{s,\frac12}}^4.
\end{aligned}
\end{equation}
In view of \eqref{eq:low result}, we can see that \eqref{eq:low result} is not affected by the choice of the maximum modulation, and hence the assumption $|\tau - \mu(n)| \gtrsim |n_2-n_3| n^3$ is enough without loss of generality.

\textbf{Case II-b.} (\emph{low $\times$ high $\times$ low $\Rightarrow$ high}) The estimation of \eqref{eq:hl0} for the integrand 
\[\wh{u}_{low}(s,n_1)\owh{u}_{med}(s,n_2)\wh{u}_{low}(s,n_3)\owh{u}(s,n)\] 
also follows \eqref{eq:low result}.

\textbf{Case II-c.} (\emph{low $\times$ high $\times$ high $\Rightarrow$ low}\footnote{\emph{high $\times$ high $\times$ low $\Rightarrow$ low} case can be estimated by the same argument due to the symmetry.}) In this case, we need to consider the following integrands
\[\wh{u}_{low}(s,n_1)\owh{u}_{high}(s,n_2)\wh{u}_{high}(s,n_3)\owh{u}(s,n)\] 
and
\[\wh{u}_{med}(s,n_1)\owh{u}_{high}(s,n_2)\wh{u}_{high}(s,n_3)\owh{u}(s,n).\] 
Since 
\[\bra{\tau - \mu(n)}^{-\frac12} \le |G|^{-\frac12 + \ep} \bra{\tau_1 - \mu(n_1)}^{-\ep} \lesssim |n_2-n_3|^{1/2}|n_2|^{-\frac32 + 4\ep} \bra{\tau_1 - \mu(n_1)}^{-\ep}\]
and 
\[|n_2|^{-\frac32 -4s + 4\ep} \lesssim N^{-\frac32 -4s + 4\ep}\]
for $-3/8 < s < 0$, \eqref{eq:low result} for both integrands still holds.

\textbf{Case II-d.} (\emph{high $\times$ low $\times$ high $\Rightarrow$ low}) The estimation of \eqref{eq:hl0} for integrands 
\[\wh{u}_{high}(s,n_1)\owh{u}_{low}(s,n_2)\wh{u}_{high}(s,n_3)\owh{u}(s,n)\] 
and
\[\wh{u}_{high}(s,n_1)\owh{u}_{med}(s,n_2)\wh{u}_{high}(s,n_3)\owh{u}(s,n)\] 
can be obtained by similar way as in \eqref{eq:low result}, since 
\[\bra{\tau - \mu(n)}^{-\frac12} \le |G|^{-\frac12 + \ep} \bra{\tau_1 - \mu(n_1)}^{-\ep} \lesssim |n_2-n_3|^{1/2}|n_3|^{-\frac32 + 4\ep} \bra{\tau_1 - \mu(n_1)}^{-\ep}\]
and 
\[|n_3|^{-\frac32 -4s + 4\ep} \lesssim N^{-\frac32 -4s + 4\ep}\]
for $-3/8 < s < 0$.

\textbf{Case III.} (\emph{high $\times$ high $\times$ low $\Rightarrow$ high}\footnote{\emph{low $\times$ high $\times$ high $\Rightarrow$ high} case can be estimated by the same argument due to the symmetry. Moreover, we can also treat \emph{high $\times$ high $\times$ high $\Rightarrow$ low} case due to the same reason in \textbf{Case II-c} and \textbf{Case II-d}.}) Under this frequency relation, since one can always get the condition \eqref{eq:hl2}, we hence obtain the same result as in \textbf{Case II}. 

By gathering the results in \textbf{Case I, II} and \textbf{III}, we can complete the proof of \eqref{eq:main}.  
\end{proof}

As an immediate result, we have the following corollary:

\begin{corollary}\label{cor:main}
Let $-1/3 \le s < 0$, $0 < T \le 1$, $t \in [-T,T]$ and $u_0 \in C^{\infty}(\T)$. Suppose that $u$ is a complex-valued smooth solution to \eqref{eq:WNLS4} on $[-T,T]$ and $u \in Y_T^{s,1/2}$. Then the following estimate holds:
\[\sup_{n \in \Z} \Big||\wh{u}(t,n)|^2-|\wh{u}_0(n)|^2\Big|\lesssim \norm{u_0}_{H^s}^4 + \big(\norm{u_0}_{H^{s}}^2 + \norm{u}_{Y_T^{s,\frac12}}^4\big)^2 +  \norm{u}_{Y_T^{s,\frac12}}^4 + \norm{u}_{Y_T^{s,\frac12}}^6.\]
\end{corollary}

Going back to \eqref{eq:linear}, we have from Proposition \ref{prop:nonres-trilinear} and Corollary \ref{cor:main} that\footnote{In view of the proof of Proposition \ref{prop:nonres-trilinear} and \eqref{eq:I}, we can extract $T^{\theta}$ from each estimate.}
\begin{equation}\label{eq:a priori 0}
\begin{aligned}
\norm{u}_{Y_{T}^{s,\frac12}} \le C\Bigg[\norm{u_{0}}_{H^s} +  T^{\gamma}\Big\{\norm{u_{0}}_{H^s}^4 +& \norm{u}_{Y_{T}^{s,\frac12}}^2 + \big(\norm{u_{0}}_{H^{s}}^2 + \norm{u}_{Y_{T}^{s,\frac12}}^4\big)^2\\
&+ \norm{u}_{Y_{T}^{s,\frac12}}^4 + \norm{u}_{Y_{T}^{s,\frac12}}^6\Big\}\norm{u}_{Y_{T}^{s,\frac12}}\Bigg]
\end{aligned}
\end{equation}
for $-\frac13 \le s < 0$.

We fix $-\frac13 \le s < 0$ and let $T$ be a positive constant with $T \le 1$ to be determined later. For given $u_0 \in H^s(\T)$, we consider a sequence $\set{u_{0,j}}$ of $C^{\infty}(\T)$ functions such that $u_{0,j} \to u_0$ in $H^s$ as $j \to \infty$. If we regard $u_{0,j}$ as initial data, we have global smooth solutions to \eqref{eq:WNLS1}. We denote these solutions by $u_j$. Then, $u_j$ satisfy the following integral equation
\[\begin{aligned}
\wh{u}_j(t,n) &= e^{it\mu_j(n)}\wh{u}_{0,j}(n)\\
&+i\int_0^te^{i(t-s)\mu_j(n)}(|\wh{u}_j(s,n)|^2-|\wh{u}_{0,j}(n)|^2)\wh{u}_j(s,n) \; ds\\
&-i\int_0^te^{i(t-s)\mu_j(n)}\sum_{\N_n}\wh{u}_j(n_1)\ol{\wh{u}_j}(n_2)\wh{u}_j(n_3) \; ds\\
\end{aligned}\]
where $\mu_j(n) = n^4 + |\wh{u}_{0,j}(n)|^2$. We denote by $\norm{\cdot}_{Y_{T,j}^{s,b}}$ the $Y_T^{s,b}$ norm associated to $\mu_j$. 

From \eqref{eq:a priori 0}, we also obtain for $u_j$ that

\begin{equation}\label{eq:a priori 1}
\begin{aligned}
\norm{u_j}_{Y_{T,j}^{s,\frac12}} \le C\Bigg[\norm{u_{0,j}}_{H^s} +  T^{\gamma}\Big\{\norm{u_{0,j}}_{H^s}^4 +& \norm{u_j}_{Y_{T,j}^{s,\frac12}}^2 + \big(\norm{u_{0,j}}_{H^{s}}^2 + \norm{u_j}_{Y_{T,j}^{s,\frac12}}^4\big)^2\\
&+ \norm{u_j}_{Y_{T,j}^{s,\frac12}}^4 + \norm{u_j}_{Y_{T,j}^{s,\frac12}}^6\Big\}\norm{u_j}_{Y_{T,j}^{s,\frac12}}\Bigg].
\end{aligned}
\end{equation}

Since $u_{0,j} \to u_0$ in $H^s$ as $j \to \infty$, we choose $K>0$ such that 
\begin{equation}\label{eq:initial}
\norm{u_{0,j}}_{H^s}, \norm{u_0}_{H^s} \le K \hspace{1em} \mbox{for all} \hspace{0.5em}j \ge 1.
\end{equation} 
Let
\[X_j(T) = \norm{u_j}_{Y_{T,j}^{s,\frac12}} (T>0), \hspace{1em} X_j(0):= \lim_{t \to 0^+}X_j(t).\]
Then, from \eqref{eq:a priori 1}, we have
\[X_j(T) \le C\Bigg[K +  T^{\gamma}\Big\{K^4 + X_j(T)^2 + \big(K^2 + X_j(T)^4\big)^2+ X_j(T)^4 + X_j(T)^6\Big\}X_j(T)\Bigg]\]
for all $j \ge 1$. Since $X_j(t)$ is continuous with respect to $t$ for smooth solutions $u_j$, by the continuity argument, we can choose $0 < T \ll 1$ such that
\begin{equation}\label{eq:a priori 2}
X_j(t) \le L 
\end{equation}
for some $L=L(K)>0$ and for $0 \le t \le T$. In view of the procedure, we can know that the choice of $T$ is independent on $j$, but dependent on $s, K$.

Now we define the projection operator $P_{\le k}$ for a positive integer $k$ by
\[P_{\le k}f = \frac{1}{\sqrt{2\pi}}\sum_{|n| \le k} \wh{f}(n)e^{inx}.\]
Let $u_{j,k} = P_{\le k}u_j$. Then $u_{j,k}$ satisfies
\[\pt \wh{u}_{j,k}(n) -i\mu_j(n) \wh{u}_{j,k}(n) = \sum_{|n| \le k}\left[i  (|\wh{u}_j(n)|^2 - |\wh{u}_{0,j}(n)|^2)\wh{u}_j(n) -i \sum_{\N_n}\wh{u}_j(n_1)\owh{u}_j(n_2)\wh{u}_j(n_3)\right]\]
with the initial data $u_{j,k}(0) = P_{\le k}u_{0,j}$.

From \eqref{eq:a priori 2}, we have 
\begin{equation}\label{eq:a priori 3}
\norm{u_{j,k}}_{Y_{T,j}^{s,\frac12}} \le \norm{u_j}_{Y_{T,j}^{s,\frac12}} \le L
\end{equation}
for all $j,k \ge 1$. 

Let $\epsilon>0$ be given. In view of the proof of  Proposition \ref{prop:main} in addition to \eqref{eq:initial} and \eqref{eq:a priori 3}, we know
\[\sum_{n \in I_N}\Big||\wh{u}_j(n)|^2-|\wh{u}_{0,j}(n)|^2\Big| \lesssim C(K,L)N^{-\delta}\]
for some $\delta \ge 0$, and where $I_N$ is defined as in \eqref{eq:I_N} for all $N \in 2^{\Z_{\ge 0}}$. Hence we obtain
\begin{equation}\label{eq:a priori 4}
\begin{aligned}
\norm{(I-P_{\le k})u_j}_{C_TH^s}^2 &\lesssim \sum_{|n| \ge k}\bra{n}^{2s}|\wh{u}_{0,j}|^2 + \sum_{N > k}N^{2s}\sum_{n \in I_N}\Big||\wh{u}_j(n)|^2-|\wh{u}_{0,j}(n)|^2\Big| \\
&\le C( \norm{(I-P_{\le k})u_{0,j}}_{H^s}^2 + k^{2s}C(K,L)),
\end{aligned}
\end{equation}
where $I$ is the identity operator. On the other hand, since $u_{0,j} \to u_0$ in $H^s$ as $j \to \infty$, there exists $N_0 > 0$ such that $\norm{u_0 - u_{0,j}}_{H^s} < \epsilon/(4C)$ holds, whenever $j > N_0$. We fix $N_0 > 0$. For each $1 \le j \le N_0$, there exist $M_j > 0$, $j=1,\cdots, N_0$ such that $\norm{(I-P_{\le k})u_{0,j}}_{H^s} < \epsilon/(2C)$ holds for $k > M_j$, $1 \le j \le N_0$. Moreover, there exists $M_0>0$ such that  $k > M_0$ implies $\norm{(I-P_{\le k})u_{0}}_{H^s} < \epsilon/(4C)$ and $k^{2s}C(K,L) < \epsilon/(2C)$. Since
\[\norm{(I-P_{\le k})u_{0,j}}_{H^s} \le \norm{u_{0,j}-u_0}_{H^s} + \norm{(I-P_{\le k})u_{0}}_{H^s},\]
we can choose $M := \max(M_0, M_j:1 \le j \le N_0)$ such that $k > M$ implies $\norm{(I-P_{\le k})u_{0,j}}_{H^s} \to 0$ for all $j \ge 1$. Thus, from \eqref{eq:a priori 4}, we have
\begin{equation}\label{eq:tightness}
\norm{(I-P_{\le k})u_j}_{H^s} < \epsilon
\end{equation}
for $k > M$ and all $j \ge 1$.

Besides, Arzel\`a-Ascoli's theorem and the diagonal argument yield that for each $N\ge1$, there exists a subsequence $\set{u_{j',j'}} \subset \set{u_{j,k}}$ (we denote $u_{j',j'}$ by $u_{j'}$) such that $\set{P_{\le N}u_{j'}}$ is a convergent sequence in $C([0,T];H^s(\T))$, i.e.
\begin{equation}\label{eq:compactness}
\norm{P_{\le N}(u_{j'} - u_{k'})}_{C([-T,T];H^s)} \to 0, \hspace{1em} j',k' \to \infty.
\end{equation}
Thus, by the $3\epsilon$-argument with \eqref{eq:tightness} and \eqref{eq:compactness}, there exists a solution $u$ to \eqref{eq:WNLS4} on $[-T,T]$ satisfying
\begin{equation}\label{eq:solution class}
u \in C([-T,T];H^s) \cap Y_T^{s,\frac12}, \hspace{1em} \norm{u}_{Y_T^{s,\frac12}} \le L, \hspace{1em} \norm{u_{j'}-u}_{C([-T,T];H^s)} \to 0 \hspace{0.3em}(j' \to \infty).
\end{equation}

\subsection{Uniqueness}
Now we consider the difference of two solutions to \eqref{eq:WNLS4} for the uniqueness part. Let $u$ and $v$ be solutions to \eqref{eq:WNLS4} with same initial data, i.e. $u_0 = v_0$, and let $w = u-v$ with $w(0) = 0$. Then, $w$ satisfies
\[\begin{aligned}
\pt\wh{w}(n) + i\mu(n)\wh{w}(n)&=i(|\wh{u}(n)|^2-|\wh{u}_0(n)^2|)\wh{w}(n) + i  (|\wh{u}(n)|^2 - |\wh{v}(n)|^2)\wh{v}(n)\\ 
&-i \sum_{\N_n}\big[\wh{w}(n_1)\owh{u}(n_2)\wh{u}(n_3)+\wh{v}(n_1)\owh{w}(n_2)\wh{u}(n_3)+\wh{v}(n_1)\owh{v}(n_2)\wh{w}(n_3)\big]\\
&=:\wh{I(u,w)}(n)+\wh{II(u,v)}(n)+\wh{III(u,v,w)}(n),
\end{aligned}\]
where $\mu(n) = n^4 + |\wh{u}_0(n)|^2$. Then, the standard $X^{s,b}$ analysis gives
\[\norm{w}_{Y^{s,1/2}_T} \lesssim \norm{I(u,w)}_{L_T^2H^s}+\norm{II(u,v)}_{L_T^2H^s}+\norm{III(u,v,w)}_{Y_T^{s,-1/2+\epsilon}}\]
for $\epsilon>0$. For $I(u,w)$ and $III(u,v,w)$, we can use the same argument as in the existence part, in particular, the estimates of $I(u)$ and $II(u)$ in the right-hand side of \eqref{eq:linear}, to obtain
\begin{equation}\label{eq:nonlinear3}
\begin{aligned}
\norm{I(u,w)}_{L_T^2H^s} &\lesssim T^{\alpha}\Big(\sup_{t \in [-T,T]}\sup_{n \in \Z}\big||\wh{u}(t,n)|^2 - |\wh{u}_0(n)|^2 \big|\Big)\norm{w}_{Y_T^{s,\frac12}}\\
&\lesssim T^{\alpha}\Big(\norm{u_0}_{H^s}^4 + \big(\norm{u_0}_{H^{s}}^2 + \norm{u}_{Y_T^{s,\frac12}}^4\big)^2 +  \norm{u}_{Y_T^{s,\frac12}}^4 + \norm{u}_{Y_T^{s,\frac12}}^6 \Big)\norm{w}_{Y_T^{s,\frac12}}\\
&\lesssim T^{\alpha}C_1(K,L)\norm{w}_{Y_T^{s,\frac12}}
\end{aligned}
\end{equation}
for some $\alpha > 0$, and
\begin{equation}\label{eq:nonlinear4}
\begin{aligned}
\norm{III(u,v,w)}_{Y_{T}^{s,-\frac12+\delta}} &\lesssim T^{\gamma}\big(\norm{u}_{Y_T^{s,\frac12}}^2 + \norm{u}_{Y_T^{s,\frac12}}\norm{v}_{Y_T^{s,\frac12}} + \norm{v}_{Y_T^{s,\frac12}}^2 \big)\norm{w}_{Y_T^{s,\frac12}}\\
&\lesssim T^{\gamma}C_3(L)\norm{w}_{Y_T^{s,\frac12}}
\end{aligned}
\end{equation}
for some $\gamma>0$, respectively, under the conditions \eqref{eq:initial} and \eqref{eq:solution class}.

For the term $II(u,v)$, similarly as the estimate of $I$ in the right-hand side of \eqref{eq:linear}, we have
\begin{equation}\label{eq:nonlinear5}
\norm{II(u,v)}_{L_T^2H^s} \lesssim T^{\beta}\Big(\sup_{t \in [-T,T]}\sup_{n \in \Z}\big||\wh{u}(t,n)|^2 - |\wh{v}(t,n)|^2 \big|\Big)\norm{v}_{Y_{T,j}^{s,\frac12}}
\end{equation}
for some $\beta>0$. As seen in \eqref{eq:nonlinear5}, the most important point of the uniqueness is how to estimate 
\begin{equation}\label{eq:diff}
\sup_{n \in \Z}\Big||\wh{u}(n)|^2 - |\wh{v}(n)|^2\Big|
\end{equation}
in $Y^{s,\frac12}$ space. In fact, since the symmetry of equation is broken in \eqref{eq:diff}, we cannot directly apply Proposition \ref{prop:main} to the difference of two solutions. However, thanks to the highly non-resonant effect from nonlinear interactions, we can overcome the lack of the symmetry in the uniqueness part. The following Proposition provides a rigorous solution to this issue:
\begin{proposition}\label{prop:main2}
Let $-1/3 \le s < 0$, $0 < T \le 1$, $t \in [-T,T]$ and $u_0 \in C^{\infty}(\T)$. Suppose that $u$ and $v$ are complex-valued smooth solutions to \eqref{eq:WNLS4} on $[-T,T]$ with $u_0 = v_0$ and $u, v \in Y_T^{s,1/2}$. Then the following estimate holds:
\begin{equation}\label{eq:main2}
\begin{aligned}
\sup_{n \in \Z} \Bigg|\mbox{Im}\Bigg[ \int_0^t  \sum_{\N_n} &\wh{v}(n_1)\owh{w}(n_2)\wh{u}(n_3)\owh{u}(n)\; ds \Bigg]\Bigg| \le C(\norm{u_0}_{H^s}, \norm{v_0}_{H^s}, \norm{u}_{Y_T^{s,\frac12}}, \norm{v}_{Y_T^{s,\frac12}}) \norm{w}_{Y_T^{s,\frac12}}.
\end{aligned}
\end{equation}
\end{proposition}

Before proving Proposition \ref{prop:main2}, we first show the following Lemma, which is the key supplement to cover the lack of the symmetry:

\begin{lemma}\label{lem:sym}
Let $-1/3 \le s < 0$, $0 < T \le 1$, $t \in [-T,T]$ and $u_0 \in C^{\infty}(\T)$. Suppose that $u$ and $v$ are complex-valued smooth solutions to \eqref{eq:WNLS4} on $[-T,T]$ with $u_0 = v_0$ and $u, v \in Y_T^{s,1/2}$. Let $w = u-v$. Then the following estimate holds:
\begin{equation}\label{eq:sym-1}
\norm{w(t)}_{H^{-\frac23}}^2 \lesssim (\norm{u}_{Y_T^{s,\frac12}}^2 + 2\norm{u}_{Y_T^{s,\frac12}}\norm{v}_{Y_T^{s,\frac12}} + \norm{v}_{Y_T^{s,\frac12}}^2)\norm{w}_{Y_T^{s,\frac12}}^2.
\end{equation}
\end{lemma}

\begin{proof}
From \eqref{eq:WNLS4}, direct computation gives
\begin{equation}\label{eq:w}
\begin{aligned}
\pt \wh{w}(n) -in^4\wh{w}(n) &= i\Big(|\wh{u}(n)|^2\wh{u}(n) - |\wh{v}(n)|^2\wh{v}(n)\Big)\\
&-i\sum_{\N_n}\Big\{\wh{w}(n_1)\owh{u}(n_2)\wh{u}(n_3)+\wh{v}(n_1)\owh{w}(n_2)\wh{u}(n_3)+\wh{v}(n_1)\owh{v}(n_2)\wh{w}(n_3)\Big\},
\end{aligned}
\end{equation}
which implies
\begin{equation}\label{eq:w energy}
\begin{aligned}
\pt|\wh{w}(n)|^2 &= 2\mbox{Im}\left[\wh{v}(n)\owh{w}(n)\wh{u}(n)\owh{w}(n) \right] \\
&-2\mbox{Im}\left[\sum_{\N_n}\big[\wh{w}(n_1)\owh{u}(n_2)\wh{u}(n_3)+\wh{v}(n_1)\owh{w}(n_2)\wh{u}(n_3)+\wh{v}(n_1)\owh{v}(n_2)\wh{w}(n_3)\big]\owh{w}(n) \right]\\
&=:A(t,n)+B(t,n).
\end{aligned}
\end{equation}
Since 
\[\sum_{n \in \Z} \bra{n}^{-\frac43}\wh{v}(n)\owh{w}(n)\wh{u}(n)\owh{w}(n) \lesssim \norm{u(t)}_{H^{-\frac13}}\norm{v(t)}_{H^{-\frac13}}\norm{w(t)}_{H^{-\frac13}}^2,\]
by the H\"older inequality and Lemma \ref{lem:sobolev}, we have
\begin{equation}\label{eq:symmetry-3}
\begin{aligned}
\int_0^t \sum_{n\in \Z} \bra{n}^{-\frac43} A(s,n) \; ds &\lesssim \norm{u}_{L_t^4H^{-\frac13}}\norm{v}_{L_t^4H^{-\frac13}}\norm{w}_{L_t^4H^{-\frac13}}^2 \\
&\lesssim \norm{u}_{Y_T^{-\frac13,\frac14}}\norm{v}_{Y_T^{-\frac13,\frac14}}\norm{w}_{Y_T^{-\frac13,\frac14}}^2.
\end{aligned}
\end{equation}
Moreover, we can apply Proposition \ref{prop:nonres-trilinear} directly for $B(t,n)$ to obtain that
\begin{equation}\label{eq:symmetry-4}
\int_0^t \sum_{n\in \Z} \bra{n}^{-\frac43} B(s,n) \; ds \lesssim (\norm{u}_{Y_T^{-\frac13,\frac12}}^2 + \norm{u}_{Y_T^{-\frac13,\frac12}}\norm{v}_{Y_T^{-\frac13,\frac12}} + \norm{v}_{Y_T^{-\frac13,\frac12}}^2)\norm{w}_{Y_T^{-\frac13,\frac12}}^2.
\end{equation}
InTogether with \eqref{eq:symmetry-3} and \eqref{eq:symmetry-4} in \eqref{eq:w energy}, we have \eqref{eq:sym-1}.
\end{proof}

Now we are ready to prove Proposition \ref{prop:main2}.

\begin{proof}[Proof of Proposition \ref{prop:main2}]
The proof is almost similar as the proof of Proposition \ref{prop:main}, while only the \textbf{case I} is slightly different. Hence it suffices to prove the case when $|n_1|\sim|n_2|\sim|n_3|\sim|n|$. We can rewrite \eqref{eq:main2} as follows:
\begin{equation}\label{eq:symmetry-1}
\begin{aligned}
&\sup_{n \in \Z} \Bigg|\mbox{Im} \int_0^t  \sum_{\N_n} \wh{v}(s,n_1)\owh{w}(s,n_2)\wh{u}(s,n_3)\owh{u}(s,n) \; ds\Bigg|\\
&=\sup_{n \in \Z}\Bigg|\mbox{Im}\int_0^t  \sum_{\N_n} e^{-is(n_1^4 - n_2^4 + n_3^4 - n^4)}\big(e^{-isn_1^4} \wh{v}(s,n_1)\big)\big(\ol{e^{-isn_2^4} \wh{w}(s,n_2)} \big) \big(e^{-isn_3^4} \wh{u}(s,n_3) \big) \big(\ol{e^{-isn^4} \wh{u}(s,n)} \big) \; ds\Bigg|.
\end{aligned}
\end{equation}
By the integration by parts with respect to the time variable $s$, the right-hand side of \eqref{eq:symmetry-1}, under the frequency relations with the same manner \eqref{eq:main1}, can be dominated by\footnote{The boundary term at $s=0$ cannot be appeared, since $w(0,x) = 0$ in this case.}
\[\begin{aligned}
&\sup_{N \ge 1}\sum_{n\in I_N}\Bigg|\sum_{\N_n} \frac{1}{iH} \wh{v}_{med}(t,n_1)\owh{w}_{med}(t,n_2)\wh{u}_{med}(t,n_3)\owh{u}(t,n)\Bigg|\\
&+\sup_{N \ge 1}\sum_{n\in I_N}\Bigg|\int_0^t \sum_{\N_n} \frac{e^{is(n_1^4 - n_2^4 + n_3^3 - n^4)}}{iH} \\
&\hspace{7em}\times \frac{d}{ds}\Big[\big(e^{-isn_1^4} \wh{v}_{med}(s,n_1)\big)\big(\ol{e^{-isn_2^4} \wh{w}_{med}(s,n_2)} \big) \big(e^{-isn_3^4} \wh{u}_{med}(s,n_3) \big) \big(\ol{e^{-isn^4} \wh{u}(s,n)} \big) \Big]\; ds\Bigg|\\
&:=\wt{I} + \wt{II}
\end{aligned}\]
where $H$ is defined as in \eqref{eq:resonant function}.

We first consider the term $\wt{I}$. In view of \eqref{eq:nonresonant estimate}, the term $\wt{I}$ can be bounded by
\[\norm{v(t)}_{H^s}\norm{u(t)}_{H^s}^2\norm{w(t)}_{H^s},\]
for $-1/2 < s$. However, due to the lack of the symmetry of the equation \eqref{eq:w}, we cannot control $\norm{w(t)}_{H^s}$ in $Y^{s,\frac12}$. Thus we use a trick in the estimate \eqref{eq:nonresonant estimate} to control $\wt{I}$. For $0 < \epsilon $, let 
\[\begin{aligned}
&\wh{g}_1(n_1) = \bra{n_1}^{-\frac13}|\wh{v}(n_1)|, \hspace{2em} \wh{g}_2(n_2) = \bra{n_2}^{-1+\epsilon}|\owh{w}(-n_2)|;\\ 
&\wh{g}_3(n_3) = \bra{n_3}^{-\frac13}|\wh{u}(n_3)|, \hspace{2em} \wh{g}_4(n) = \bra{n}^{-\frac13}|\owh{u}(-n)|.
\end{aligned}\]
Then, from the fact that
\[\frac{1}{|n_1-n_2||n|^2} \le \frac{1}{|n_1-n_2|}|n_1n_3n|^{-1/3}|n_2|^{-1+\epsilon}N^{-\ep},\]
we obtain
\[\begin{aligned}
\wt{I} &\lesssim \sup_{N\ge1}N^{-\ep}\sum_{n, \N_{n}} \frac{1}{|n_1 -n_2|} \wh{g}_1(n_1)\wh{g}_2(-n_2)\wh{g}_3(n_3)\wh{g}_4(-n) \\
&\le \sup_{N\ge1}N^{-\ep}\sum_{n_1,n, n' \neq 0} \frac{1}{|n'|}\wh{g}_1(n_1)\wh{g}_2(n'-n_1)\wh{g}_3(n-n')\wh{g}_4(-n)\\
&=\sup_{N\ge1}N^{-\ep}\sum_{0 < |n'| \le N} \frac{1}{|n'|}\wh{g_1g_2}(n')\wh{g_3g_4}(-n')\\
&\lesssim \sup_{N\ge1}N^{-\ep}\log N\norm{\wh{g_1g_2}}_{\ell^{\infty}}\norm{\wh{g_3g_4}}_{\ell^{\infty}}\\
&\lesssim \norm{v(t)}_{H^{-\frac13}}\norm{u(t)}_{H^{-\frac13}}^2\norm{w(t)}_{H^{-1+\epsilon}}.
\end{aligned}\]
We choose $0 < \epsilon \le \frac13$ to satisfy $\norm{w(t)}_{H^{-1 + \epsilon}} \lesssim \norm{w(t)}_{H^{-\frac23}}$. By  Corollary \ref{cor:trilinear} for $\norm{v(t)}_{H^{-\frac13}}$ and $\norm{u(t)}_{H^{-\frac13}}$ and Lemma \ref{lem:sym} for $\norm{w(t)}_{H^{-1+\epsilon}}$, we obtain
\begin{equation}\label{eq:1}
\wt{I} \lesssim C(\norm{u_0}_{H^s}, \norm{v_0}_{H^s}, \norm{u}_{Y_T^{s,\frac12}}, \norm{v}_{Y_T^{s,\frac12}}) \norm{w}_{Y_T^{s,\frac12}}.
\end{equation}

Now we consider the term $\wt{II}$. The estimate of $\wt{II}$ can be obtained by using the same way as the estimate of $II$ in \textbf{Case I} in the proof of Proposition \ref{prop:main}. Indeed, if we take the time derivative at the solution $u$ or $v$, $\wt{II}$ can be reduced similarly as \eqref{eq:hhhh2}. Otherwise, from the following observation 
\[\begin{aligned}
\pt \left( e^{-itn^4}\wh{w}(n) \right) &= e^{-itn^4} \left( \pt \wh{w}(n) - in^4 \wh{w}(n) \right)\\
&= e^{-itn^4} \Bigg( i\Big\{\wh{w}(n)\owh{u}(n)\wh{u}(n) + \wh{v}(n)\owh{w}(n)\wh{u}(n) + \wh{v}(n)\owh{v}(n)\wh{w}(n)\Big\} \\
&- i\sum_{\N_n}\Big\{\wh{w}(n_1)\owh{u}(n_2)\wh{u}(n_3) + \wh{v}(n_1)\owh{w}(n_2)\wh{u}(n_3) + \wh{v}(n_1)\owh{v}(n_2)\wh{w}(n_3)\Big\} \Bigg),
\end{aligned}\]
we can still apply the same argument as the estimate of \eqref{eq:hhhh2}. Hence, we have
\begin{equation}\label{eq:2}
\wt{II} \lesssim C(\norm{u_0}_{H^s}, \norm{v_0}_{H^s}, \norm{u}_{Y_T^{s,\frac12}}, \norm{v}_{Y_T^{s,\frac12}}) \norm{w}_{Y_T^{s,\frac12}}.
\end{equation}
Together with \eqref{eq:1} and \eqref{eq:2}, we complete the proof of Proposition \ref{prop:main2}.
\end{proof}

As an immediate consequence of Proposition \ref{prop:main2}, we have the following corollary:

\begin{corollary}\label{cor:main2}
Let $-1/3 \le s < 0$, $0 < T \le 1$, $t \in [-T,T]$ and $u_0 \in C^{\infty}(\T)$. Suppose that $u$ and $v$ are complex-valued smooth solutions to \eqref{eq:WNLS4} on $[-T,T]$ with $u_0 = v_0$ and $u, v \in Y_T^{s,1/2}$. Then the following estimate holds:
\begin{equation}\label{eq:main3}
\begin{aligned}
\sup_{n \in \Z} \Big||\wh{u}(n)|^2-|\wh{v}(n)|^2\Big| \le C(\norm{u_0}_{H^s}, \norm{v_0}_{H^s}, \norm{u}_{Y_T^{s,\frac12}}, \norm{v}_{Y_T^{s,\frac12}}) \norm{w}_{Y_T^{s,\frac12}}.
\end{aligned}
\end{equation}
\end{corollary}

\begin{proof}
We observe from \eqref{eq:res-nonres} that
\begin{equation}\label{eq:difference}
\begin{aligned}
\sup_{n \in \Z}\big||\wh{u}(n)|^2 - |\wh{v}(n)|^2\big|&\\
=\sup_{n \in \Z} \Bigg|2\mbox{Im}\Bigg[ \int_0^t  \sum_{\N_n} &\Big\{\wh{w}(n_1)\owh{u}(n_2)\wh{u}(n_3)\owh{u}(n)+\wh{v}(n_1)\owh{w}(n_2)\wh{u}(n_3)\owh{u}(n)\\
&+\wh{v}(n_1)\owh{v}(n_2)\wh{w}(n_3)\owh{u}(n) + \wh{v}(n_1)\owh{v}(n_2)\wh{v}(n_3)\owh{w}(n)\Big\} \; ds \Bigg]\Bigg|.
\end{aligned}
\end{equation}
Proposition \ref{prop:main2} can be directly applied to not only the second term, but also the other terms in the right-hand side of \eqref{eq:difference}, while the dependence of the constant $C$ is slightly different. Hence we obtain \eqref{eq:main3}.
\end{proof}

Hence, by Corollary \ref{cor:main2} with \eqref{eq:initial} and \eqref{eq:solution class}, we obtain
\begin{equation}\label{eq:nonlinear6}
\begin{aligned}
\norm{II(u,v)}_{L_T^2H^s} &\lesssim T^{\beta}\Big(\sup_{t \in [-T,T]}\sup_{n \in \Z}\big||\wh{u}(t,n)|^2 - |\wh{v}(n)|^2 \big|\Big)\norm{v}_{Y_{T,j}^{s,\frac12}}\\
&\lesssim T^{\beta}C_2(K,L)\norm{w}_{Y_T^{s,\frac12}}.
\end{aligned}
\end{equation}

Thus, by gathering \eqref{eq:nonlinear3}, \eqref{eq:nonlinear4} and \eqref{eq:nonlinear6}, we can choose $T'>0$ small enough to have
\[\norm{w}_{Y_{T'}^{s,\frac12}} \le c\norm{w}_{Y_{T'}^{s,\frac12}},\]
for some $0 < c < 1$, which implies $w \equiv 0$ on $[-T',T']$. By repeating this procedure, we can obtain the uniqueness of the solution on $[-T,T]$.

\subsection{Continuity of the flow map}
The continuous dependence of the solution with respect to the initial data can be shown by the same way to show the existence part. Only different thing is that we do not need to extract a subsequence in the limiting process due to the uniqueness part. We omit the details and thus we complete the proof of Theorem \ref{thm:main}.

\section{Symplectic Non-squeezing}\label{sec:nonsqueezing}
In this section, we prove the non-squeezing property of \eqref{eq:NLS1}. We follow the argument in Colliander, Keel, Staffilani, Takaoka and Tao \cite{CKSTT2005} in the context of KdV flow. We first state the finite dimensional non-squeezing property of \eqref{eq:truncated equation} as an application of Gromov's non-squeezing theorem.
\begin{lemma}\label{lem:Nonsqueezing of trun. flow}
Let $N \geq 1$, $ 0 <r <R$, $u_* \in P_{\le N} L^2(\T)$, $0< |n_0| \leq N$, $ z \in \C$ and $T>0$. Let $\Phi^N(t) :  P_{\le N} L^2(\T)  \to P_{\le N} L^2(\T)$ be the flow map to \eqref{eq:truncated equation}.
Then 
\begin{equation*}
\Phi^N(T)(B_R^N ( u_* )) \not \subseteq C_{n_0,r}^N(z),
\end{equation*}
where $B_R^N (u_{\ast})$ and $C_{n_0,r}^N(z)$ are finite dimensional restriction of a ball and a cylinder defined as follows:
\[\begin{aligned}
&B^{N}_R (u_{\ast}) := \set{u \in P_{\le N}L^2(\T) : \norm{u-u_{\ast}}_{L^2} \le R}, \\
&C^{N}_{n,r} (z) := \set{u \in P_{\le N}L^2(\T) : |\wh{u}(n) - z| \leq r}.
\end{aligned}\]
\end{lemma}

The main task in this section is to show that Lemma \ref{lem:Nonsqueezing of trun. flow} still holds true when $N \to \infty$. To do this, it is enough to show the following proposition, which makes hold that the flow map of \eqref{eq:truncated equation} approximates to the original flow map of \eqref{eq:NLS1} in the strong $L^2$-topology.

\begin{proposition}\label{prop:approximation}
Let $T>0$ and $N\gg1$. Let $u_0 \in L^2$ with frequency support $|n| \le N$. Then we have
\[\sup_{|t| \le T} \norm{P_{\le N^{1/2}}(\Phi(t)u_0 - \Phi^N(t)u_0)}_{L^2} \leq C\Big(T,\norm{u_0}_{L^2}\Big)N^{-\sigma}\]
for some $\sigma >0$.
\end{proposition}

Proposition \ref{prop:approximation} is deduced by the following two lemmas: 

\begin{lemma}\label{lem:approx1}
Let $N'\ge1$ be fixed. Let $u_0, \ul{u_0} \in L^2(\T)$ with $P_{\le N'} u_0 = P_{\le N'}\ul{u_0}$ \footnote{Similarly as the argument in Bourgain's work \cite{Bourgain1994}, the high frequency perturbation condition can be replaced by $\norm{u_0}_{L^2(\T)}=\norm{\ul{u_0}}_{L^2(\T)}$ thanks to the mass conservation law \eqref{eq:mass}.}. Then, if $T'>0$ is sufficiently small depending on $\norm{u_0}_{L^2(\T)}$ and $\norm{\ul{u_0}}_{L^2(\T)}$, we have
\[\sup_{|t| \le T'} \norm{P_{\le N' - (N')^{\frac12}}(\Phi(t)u_0 - \Phi(t)\ul{u_0})}_{L^2} \leq C\Big(T,\norm{u_0}_{L^2}, \norm{\ul{u_0}}_{L^2}\Big) N'^{-\sigma}\]
for some $\sigma>0$.
\end{lemma}

\begin{lemma}\label{lem:approx2}
Let $N' \ge 1$ be fixed. Let $u_0 \in L^2$ with frequency support $|n| \le N'$. Then, if $T'>0$ is sufficiently small depending on $\norm{u_0}_{L^2(\T)}$, we have
\[\sup_{|t| \le T'} \norm{P_{\le (N')^{1/2}}(\Phi(t)u_0 - \Phi^{N'}(t)u_0)}_{L^2} \leq C\Big(T',\norm{u_0}_{L^2}\Big)N'^{-\sigma}\]
for some $\sigma >0$.
\end{lemma}

\begin{remark}
We note in Lemma \ref{lem:approx1} that the low frequencies are stable on the flow $\Phi$ for high frequency perturbations of data, and in Lemma \ref{lem:approx2} that $\Phi^N$ approximates to $\Phi$ at low frequencies. 
\end{remark}

\begin{remark}\label{rem:deduction}
To deduce Proposition \ref{prop:approximation}, once we have Lemmas \ref{lem:approx1} and \ref{lem:approx2}, by dividing the whole time interval $[-T,T]$ into a finite number $M=M(T', \norm{u_0}_{L^2}, \norm{\ul{u_0}}_{L^2},s)$ of time intervals\footnote{Here $T'= T'(\norm{u_0}_{L^2},\norm{\ul{u_0}}_{L^2})>0$ is obtained from the standard local theory in $L^2$, and $T'$ in Lemmas \ref{lem:approx1} and \ref{lem:approx2} would suffice to be smaller than this $T'$.}, we can show that Proposition \ref{prop:approximation} holds on each interval iteratively, and hence so on the whole interval $[-T,T]$.
\end{remark}

We first prove Lemma \ref{lem:approx1}. We consider the following equation for the simplicity.
\begin{equation}\label{eq:NLS5}
u_t - i \px^4u = \N(u,u,u),
\end{equation}
where $\N(u,v,w) = \mu i u\ol{v}w$. We also use notations $N_R$ and $N_{NR}$ defined as in \eqref{eq:res. nonl.} and \eqref{eq:nonres. nonl.}, respectively.

\begin{proof}[Proof of Lemma \ref{lem:approx1}]
Let $u$ and $\ul{u}$ be solutions to \eqref{eq:NLS5} on the interval $[-T',T']$. From the local well-posedness theory, we have the local estimates
\begin{equation}\label{eq:small}
\norm{u}_{Z^{0,\frac12}} + \norm{\ul{u}}_{Z^{0, \frac12}} \le C,
\end{equation}
by choosing the sufficiently small time $T'$ depending on the $L^2$-norms of $u_0$ and $ \ul{u_0}$.

We split the solutions $u$ and $\underline u$ into the two portions using the following argument: Let $M \in \left[N' - \left(N'\right)^{\frac{1}{2}}, N'\right]$ be an integer. We separate $u$ as
\[u = u_{lo} +u_{hi},\]
where 
\[u_{lo} := P_{\le M}u, \hspace{1em} u_{hi}:= (1-P_{\le M})u.\]
From \eqref{eq:small}, we have
\begin{equation}\label{eq:split}
\norm{u_{lo}}_{Z^{0,\frac12}}, \hspace{0.1em} \norm{u_{hi}}_{Z^{0,\frac12}} \le C.
\end{equation}
We also split $\ul{u}$ and obtain the similar result as \eqref{eq:split} for $\underline u$. We apply $P_{\le M}$ to \eqref{eq:NLS5} to see that $u_{lo}$ obeys the equation
\begin{equation}\label{eq:low freq eq}
(\pt + \px^3)u_{lo} = P_{\le M}\N_{R}(u,u,u) + P_{\le M}\N_{NR}(u,u,u).
\end{equation}

In order to control the right-hand side of \eqref{eq:low freq eq}, we first define the \emph{error terms} to be any quantity with $Z^{0,-\frac12}$-norm of $O((N')^{-\sigma})$ in $\N_{NR}(u,u,u)$. Thanks to the $N_{max}^{-\frac12 +}$ factor in trilinear estimates \eqref{eq:bilinear-1} and \eqref{eq:bilinear-2}, all terms except for $P_{\le M}\N_{NR}(u_{lo},u_{lo},u_{lo})$ in $P_{\le M}\N_{NR}(u,u,u)$ are error terms. Thus, we conclude that $u_{lo}$ obeys the equation
\begin{equation}\label{eq:error1}
(\pt - i\px^4)u_{lo} = P_{\le M}\N_{R}(u, u, u) + P_{\le M}\N_{NR}(u_{lo},u_{lo},u_{lo}) + \mbox{error term}.
\end{equation} 
In the same manner, the function $\ul{u_{lo}}$ also obeys the same equation (with slightly different error terms)
\begin{equation}\label{eq:error2}
(\pt - i\px^4)\ul{u_{lo}} = P_{\le M}\N_{R}(\ul{u}, \ul{u}, \ul{u}) + P_{\le M}\N_{NR}(\ul{u_{lo}},\ul{u_{lo}},\ul{u_{lo}}) + \mbox{error term}.
\end{equation}
 
Now we consider $P_{\le M}\N_{R}(u, u, u) - P_{\le M}\N_{R}(\ul{u}, \ul{u}, \ul{u})$. From the $L^2$-conservation law and a simple calculation, we know that
\[\begin{aligned}
P_{\le M}&\N_{R}(u, u, u) - P_{\le M}\N_{R}(\ul{u}, \ul{u}, \ul{u})\\ 
&=\frac{i}{\sqrt{2\pi}}\sum_{|n| \le M}e^{inx}\left((\wh{u}(n)-\wh{\ul{u}}(n))|\wh{u}(n)|^2+\wh{\ul{u}}(n)(\ol{\wh{u}}(n)-\ol{\wh{\ul{u}}}(n))\wh{u}(n)+|\wh{\ul{u}}(n)|^2(\wh{u}(n)-\wh{\ul{u}}(n))\right)\\
&-\frac{2i}{\sqrt{2\pi}}\sum_{|n| \le M}e^{inx}(\norm{u_0}_{L^2}^2 - \norm{\ul{u_0}}_{L^2}^2)(\wh{u}(n)-\wh{\ul{u}}(n)),
\end{aligned}\]
and this in addition to \eqref{eq:resonant} and \eqref{eq:split} implies
\[\norm{P_{\le M}\N_{R}(u, u, u) - P_{\le M}\N_{R}(\ul{u}, \ul{u}, \ul{u})}_{Z^{0,-\frac12}} \le C'\norm{u_{lo}-\ul{u_{lo}}}_{Z^{0,\frac12}},\]
where $C'$ depends on $T', \norm{u_0}_{L^2}, \norm{\ul{u_0}}_{L^2}, \norm{u_{lo}}_{Z^{0,\frac12}}$, and $\norm{\ul{u_{lo}}}_{Z^{0,\frac12}}$.
 
Hence, since $u_{lo}(0) = \ul{u_{lo}}(0)$, we have from the standard local well-posedness theory in addition to a suitable small $T'$ and \eqref{eq:split} that
\[\norm{u_{lo} - \ul{u_{lo}}}_{Z^{s,\frac12}} \lesssim \left(N'\right)^{-\sigma},\]
which implies Lemma \ref{lem:approx1} by $Z^{s,\frac12} \subset C_{t}H^s$.
\end{proof}

The proof of Lemma \ref{lem:approx2} is quite similar as the proof of Lemma \ref{lem:approx1}. The key in the proof of Lemma \ref{lem:approx1} is to construct \eqref{eq:error1} and \eqref{eq:error2}. Due to the fact that $P_{\le N}P_{\le 2N} = P_{\le N}$, we need to obtain 
\[(\pt - i\px^4)u_{lo} = P_{\le M}\N_{R}(u, u, u) + P_{\le M}\N_{NR}(u_{lo},u_{lo},u_{lo}) + \mbox{error term}\]
and
\[(\pt - i\px^4)v_{lo} = P_{\le M}\N_{R}(v, v, v) + P_{\le M}\N_{NR}(v_{lo},v_{lo},v_{lo}) + \mbox{error term}.\]
However, the similar argument in the proof of Lemma \ref{lem:approx1} can be applied here, so we omit the detailed proof of Lemma \ref{lem:approx2}. Hence, Lemmas \ref{lem:approx1} and \ref{lem:approx2} in addition to Remark \ref{rem:deduction} implies Proposition \ref{prop:approximation}. The following is an immediate corollary of Proposition \ref{prop:approximation}:

\begin{corollary}\label{cor:main approx}
Let $n_0 \in \Z$, $T>0$, $A>0$ and $ 0 < \varepsilon \ll 1$. Then there exists a frequency $N_0 = N_0(n_0, T, \varepsilon, A)> |n_0|$ such that
\begin{equation*}
\left|\ft[\Phi(T)u_0 - \Phi^N(T)u_0](n_0)\right| \ll \varepsilon
\end{equation*}
for all $N \gg N_0$ and all $u_0 \in B_A^N(0) \subset  P_{\le N} L^2(\T)$.
\end{corollary}

Now we are ready to prove Theorem \ref{thm:Nonsqueezing thm}.
\begin{proof}[Proof of Theorem \ref{thm:Nonsqueezing thm}]
For given $r,R > 0$ with $0< r < R$, we choose $0 < \varepsilon < \frac{R-r}{2}$. We also choose $A>0$ such that the ball $B^{\infty}_{A}(0)$ contains the ball $B^{\infty}_{R}(u_*)$. Since $u_* \in L^2$, we can choose $N > N_0(T,\varepsilon,n_0, A) \gg 1$ such that
\begin{equation}\label{eq:high}
\norm{(1-P_{\le N})u_*}_{L^2} \le \varepsilon.
\end{equation}
Since our choice of $\varepsilon$ still satisfies $r+\varepsilon < R-\varepsilon$, we apply Lemma \ref{lem:Nonsqueezing of trun. flow} to find initial data $u_{0,N} \in P_{\le N} L^2(\T)$ satisfying
\begin{equation}\label{eq:cut ball}
\norm{u_{0,N} - P_{\le N}u_*}_{L^2} \le R - \varepsilon
\end{equation}
and
\begin{equation}\label{eq:cut ball2}
\left|\ft_x[\Phi^N(T)u_{0,N}](n_0)-z\right| > r+\varepsilon.
\end{equation}
From \eqref{eq:high} and \eqref{eq:cut ball}, the triangle inequality yields
\[\norm{u_{0,N} - u_*}_{L^2} \le R.\]
For $u_0 := u_{0,N}$, we can conclude from the triangle inequality, \eqref{eq:cut ball2} and Corollary \ref{cor:main approx} that
\[\begin{aligned}
&\left|\ft_x[\Phi(T)u_0](n_0) - z\right| \\
\ge& \left|\ft_x[\Phi^N(T)u_{0,N}](n_0) - z\right| - \left|\ft_x[\Phi(T)u_0](n_0) - \ft_x[\Phi^N(T)u_0](n_0)\right| \\
 >& r+\varepsilon -\varepsilon = r,
\end{aligned}\]
which completes the proof.
\end{proof}

\end{document}